\documentclass[12pt,a4paper]{article}
\usepackage{graphicx}
\usepackage[english]{babel}
\usepackage{amsfonts,amsmath,amssymb,amsthm}
\usepackage[utf8]{inputenc}
\usepackage{srcltx}
\usepackage{latexsym,,longtable,mathrsfs,array,tabularx,stmaryrd}
\usepackage[unicode,final,hyperindex,hypertex]{hyperref}
\usepackage{bbm}
\usepackage{comment}
\usepackage{cite}
\usepackage{color}

\usepackage{mathabx}
\usepackage{enumerate}
\usepackage[centerlast, small]{caption}

\newtheorem{thm}{Theorem}
\newtheorem{lem}{Lemma}

\newcommand{\diag}{{\mathrm{diag}}}
\newcommand{\F}{\mathbbmss{F}}
\newcommand{\splitext}{\,\colon\!}
\newcommand{\arbitraryext}{\,\ldotp}
\newcommand{\GL}{{\mathrm{GL}}}

\newcommand{\PSL}{\mathop{\rm PSL}\nolimits}           
\newcommand{\PSp}{\mathop{\rm PSp}\nolimits}           
\newcommand{\SL}{\mathop{\rm SL}\nolimits}           
\newcommand{\SU}{\mathop{\rm SU}\nolimits}             
\newcommand{\PSU}{\mathop{\rm PSU}\nolimits}  
\newcommand{\PGU}{\mathop{\rm PGU}\nolimits}
\newcommand{\PGL}{{\mathrm{PGL}}}
\renewcommand{\P}{{\rm P}}
\newcommand{\GU}{\mathop{\rm GU}\nolimits}

\newcommand{\Aut}{\operatorname{Aut}}
\newcommand{\Out}{\operatorname{Out}}
\renewcommand{\unlhd}{\trianglelefteqslant}
\newcommand{\ld}{\mathop{\ldbrack}}
\newcommand{\rd}{\mathop{\rdbrack}}

\newcommand{\D}{\mathcal{D}}
\renewcommand{\C}{\mathcal{C}}
\newcommand{\E}{\mathcal{E}}
\renewcommand{\U}{\mathcal{U}}
\renewcommand{\leftthreetimes}{}

\newcommand{\lcm}{\mathrm{lcm}}

\begin{document}

\title{The Hall property $\D_\pi$ is inherited by overgroups of
$\pi$-Hall subgroups\thanks{The work is supported by  Russian Science
Foundation (project 14-21-00065)} } 


\author{Nomina Ch. Manzaeva  \\ Novosibirsk State University, Novosibirsk, Russian
Federation \\
Sobolev Institute of Mathematics, Novosibirsk, Russian Federation \\
              {manzaeva@mail.ru}      \and
        Danila O. Revin \\   Sobolev Institute of Mathematics, Novosibirsk, Russian Federation \\
Novosibirsk State University, Novosibirsk, Russian Federation \\
{revin@math.nsc.ru}    \and
        Evgeny P. Vdovin Sobolev\\ Institute of Mathematics, Novosibirsk, Russian Federation \\
Novosibirsk State University, Novosibirsk, Russian Federation \\
{vdovin@math.nsc.ru}
}


\maketitle

\begin{abstract}
 Let $\pi$ be a set of primes. We say that a finite group $G$ is a
$\D_\pi$-group if the maximal $\pi$-subgroups of $G$ are conjugate. In this
paper, we give an affirmative answer to Problem 17.44(b) from ``Kourovka
notebook'', namely we prove that in a $\D_\pi$-group an overgroup of a
$\pi$-Hall subgroup is always a $\D_\pi$-group. 
\end{abstract}


\section{Introduction}

Throughout $G$~is  a finite group, and $\pi$~ is a set of primes.  We denote
by $\pi'$ the set of all primes not in $\pi$, by $\pi(n)$ the set of all
prime divisors of a positive integer~$n$, given a group $G$ we denote
$\pi(|G|)$ by $\pi(G)$. A natural  number $n$ with $\pi(n)\subseteq \pi$ is
called a {\em $\pi$-number}, while  a group $G$ with $\pi(G)\subseteq \pi$ is
called a {\it $\pi$-group}. A subgroup $H$ of $G$ is called a {\it $\pi$-Hall
subgroup}, if $\pi(H)\subseteq \pi$ and $\pi(|G:H|)\subseteq \pi'$, i.e. the
order of $H$ is a $\pi$-number and the index of $H$ is a $\pi'$-number.

Following \cite{Hall1956}, we say that $G$ {\it satisfies $\E_\pi$} (or
briefly $G\in \E_\pi$), if $G$ has a $\pi$-Hall subgroup. If $G$ satisfies
$\E_\pi$ and every two $\pi$-Hall subgroups of $G$ are conjugate, then we say
that $G$ {\it satisfies $\C_\pi$} ($G\in \C_\pi$). Finally, $G$ {\it
satisfies~$\D_\pi$} ($G\in \D_\pi$), if $G$ satisfies $\C_\pi$ and every
$\pi$-subgroup of $G$ is included in a $\pi$-Hall subgroup of~$G$. Thus $G\in
\D_\pi$ if a complete analogue of the Sylow theorems for $\pi$-subgroups of
$G$ holds. Moreover, the Sylow theorems imply that $G\in\D_\pi$ if and only
if the maximal $\pi$-subgroups of $G$ are conjugate. A group  $G$ satisfying
$\E_\pi$ ($\C_\pi$, $\D_\pi$) is also called an {\it$\E_\pi$-group}
(respectively, a $\C_\pi$-group, a $\D_\pi$-group). Given set $\pi$ of primes
 we denote by $\E_\pi$, $\C_\pi$, and $\D_\pi$ the classes of all finite
$\E_\pi$-, $\C_\pi$-, and $\D_\pi$- groups, respectively.

 In the paper,  we solve the following problem from
``Kourovka notebook''~\cite{Kour}:

\medskip

\noindent{\bfseries Problem 1.}\label{17.44b}{\cite[Problem 17.44(b)]{Kour}}
 In a $\D_\pi$-group, is an overgroup of a $\pi$-Hall subgroup always a
$\D_\pi$-group?

\medskip

 The  analogous problem for $\C_\pi$-property (see \cite[Problem~17.44(a)]{Kour}) was answered in  the affirmative
(cf. \cite{VdovinRevin2012, VdovinRevin2013}). An equivalent formulation to
this statement is: {\em in a $\C_\pi$-group $\pi$-Hall subgroups are
pronormal.} Recall that a subgroup $H$ of a group $G$ is said to be {\em
pronormal} if, for every $g\in G$, $H$ and $H^g$ are conjugate in $\langle H,
H^g\rangle$.

According to \cite{VdovinRevin2011}, we say that $G$ {\it satisfies
$\U_\pi$}, if $G\in \D_\pi$ and every overgroup of a $\pi$-Hall subgroup of
$G$ satisfies $\D_\pi$. We denote also by $\U_\pi$ the class of all finite
groups satisfying $\U_\pi$. Thus Problem~\ref{17.44b} can be reformulated in
the following way:

\medskip

\noindent{\bfseries Problem 2.}  Is it true that $\D_\pi=\U_\pi$?

\medskip

The following main theorem gives an affirmative answer to Problems~1 and~2.

\begin{thm}{\em(Main theorem)}\label{main}
Let $\pi$~be a set of primes.  Then  $\D_\pi=\U_\pi$. In other words, if $G$
satisfies $\D_\pi$ and $H$ is a $\pi$-Hall subgroup of $G$, then every
subgroup~$M$ of $G$ with $H\le M$ satisfies~$\D_\pi$.
\end{thm}

One can formulate this statement by using the concept of strong pronormality.
According to \cite{VdovinRevin2011}, a subgroup $H$ of a group $G$ is said to
be {\em strongly pronormal} if, for every $g\in G$ and $K\le H$, there exists
$x\in \langle H, K^g\rangle$ such that $K^{gx}\le H$. Theorem \ref{main} is
equivalent to the following

\begin{thm}\label{main2}
Let $\pi$~be a set of primes. In a $\D_\pi$-group $\pi$-Hall subgroups are
strongly pronormal.
\end{thm}

 In \cite[Theorem~7.7]{RevinVdovin2006}, it was  proven that  $G$
satisfies~$\D_\pi$ if and only if each composition factor of $G$ satisfies
$\D_\pi$. Using this result, an analogous criterion for $\U_\pi$ is obtained
in \cite{VdovinManRevin2012}.

 \begin{thm}\label{reduction}{\em \cite[Theorem~2]{VdovinManRevin2012}}
A finite group $G$ satisfies $\U_\pi$ if and only if each composition factor
of $G$ satisfies~$\U_\pi$.
 \end{thm}

In order to solve \cite[Problem~17.44(a)]{Kour}, the pronormality of Hall
subgroup in finite simple groups was proven in~\cite{VdovinRevin2012}. The
strong pronormality of Hall subgroups in finite simple groups together with
Theorem \ref{reduction} would imply the main theorem. However, M.Nesterov in
\cite{Nesterov} showed that $\PSp_{10}(7)$ contains a   $\{2,3\}$-Hall
subgroup that is not strongly pronormal.

Theorem \ref{reduction} reduces Problem 1  to a similar problem for simple
$\D_\pi$-groups. All simple $\D_\pi$-groups are known: in pure arithmetic
terms, necessary and sufficient conditions for a simple group $G$ to satisfy
$\D_\pi$ can be found in \cite{Revin2008e}. It was proved in
\cite{VdovinManRevin2012} that if $G\in \D_\pi$ is an alternating group, a
sporadic group or a group of Lie type in characteristic $p\in \pi$, then $G$
satisfies~$\U_\pi$. An affirmative answer to Problem~\ref{17.44b} in case
$2\in \pi$ is obtained in \cite{Manzaeva2014}. In this paper, we consider
the remaining case of  $\D_\pi$-groups of Lie type in characteristic~$p$
with~${2,p\notin \pi}$.

\section{Notation and preliminary results}

All groups in the paper are assumed to be finite. Our notation is standard
and agrees with that of \cite{CFSG} and \cite{KleiLie}. By $A\splitext B$ and
$A\arbitraryext B$ we denote a split extension and an arbitrary extension of
a group $A$ by a group $B$, respectively. Symbol $A\times B$  denotes the
direct  product of $A$ and $B$. If $G$ is a group and $S$ is a permutation
group, then $G\wr S$ is the permutation wreath product of $G$ and~$S$. We use
notations $H\leq G$ and $H \unlhd G$ instead of ``$H$ is a subgroup of~$G$''
and ``$H$ is a normal subgroup of~$G$'', respectively. For $M\subseteq G$ we
set $M^G=\{M^g\mid g\in G\}$. The subgroup generated by a subset $M$ is
denoted by $\langle M \rangle$. The normalizer and the centralizer of $H$ in
$G$ are denoted by $N_G(H)$ and $C_G(H)$, respectively, while $Z(G)$ is the
center of~$G$.  The generalized Fitting subgroup of $G$ is denoted by
$F^*(G)$. For a group $G$, we denote by $\Aut(G)$ and $\Out(G)$ the
automorphism group and the outer automorphism group, respectively. Denote a
cyclic group of order $n$ by $n$,  and an arbitrary solvable group of order
$n$ by $\ldbrack n \rdbrack$. Recall that, for a group $X$ and a prime $t$, a
$t$-rank $m_t(X)$ is the maximal rank of elementary abelian $t$-subgroups of
$X$.

Throughout, $\F_q$ is a finite field of order $q$ and characteristic $p$.
By $\eta$ we always denote an element of the set $\{+, -\}$ and we use $\eta$
instead of $\eta1$ as well. In order to make uniform statements and
arguments, we use the following notations $\GL_n^+(q)=\GL_n(q)$,
$\GL_n^-(q)=\GU_n(q)$, $\SL_n^+(q)=\SL_n(q)$, $\SL_n^-(q)=\SU_n(q)$,
$\PSL_n^+(q)=\PSL_n(q)$, $\PSL_n^-(q)=\PSU_n(q)$, $E_6^+(q)=E_6(q)$,
$E_6^-(q)={}^2E_6(q)$.   If $G$ is a group of Lie type, then by $W(G)$ we
denote the Weyl group of $G$.

 The integral part of a real number $x$ is denoted by   $\left[ x \right]$. For
integers $n$ and  $m$, we denote
 by $\gcd(n,m)$ and $\lcm(n,m)$ the greatest common divisor and the least common
multiple, respectively. If $\pi$ is a set of primes,
 then $\min(\pi)$ is the smallest prime in $\pi$. If  $n$~is a positive integer,
then $n_\pi$ is the largest divisor $d$ of $n$
 with $\pi(d)\subseteq\pi$.  If $g$ is an element of a group then there are
elements $g_\pi$ and $g_{\pi'}$ in $\langle g \rangle$
 such that $g=g_\pi g_{\pi'}$ and $|g_\pi|$ is a $\pi$-number, while
$|g_{\pi'}|$ is a $\pi'$-number.

 If $r$~is an odd prime and  $k$~is an integer not divisible by $r$, then
$e(k,r)$ is the smallest
 positive integer  $e$ with $k^e \equiv 1 \pmod r$.
 So, $e(k,r)$ is the multiplicative order of $k$ modulo~$r$. In particular, if
$e=e(k,r)$, then
 $$
 e(k^a,r)=\frac{e}{\gcd(e,a)}.
 $$
For a natural number $e$ set

 $$e^*=\left\{
\begin{array}{ll}
2e & \text{ if }  e\equiv 1\pmod 2,\\
e & \mbox{ if   }  e\equiv 0\pmod 4,\\
e/2 & \mbox{ if  }  e\equiv 2\pmod 4.
\end{array}
\right.
$$

The next result may be found in \cite{Weir1955}.

\begin{lem}\label{r4ast'}{\em (\cite{Weir1955}, \cite[Lemmas~2.4 and
2.5]{Gross1995})}
 Let $r$ be an odd prime, $k$~an integer not divisible by $r$,  and  $m$~a
positive integer. Denote $e(k,r)$ by $e$.

 Then the following identities hold.

$$(k^m-1)_r=\left\{
\begin{array}{ll}
(k^e-1)_r (m/e)_{r} & \text{ if } e \text{ divides } m,\\
1 & \text{ if } e \text{ does not divide } m;\end{array}\right.$$

$$(k^m-(-1)^m)_r=\left\{
\begin{array}{ll}
(k^{e^*}-(-1)^{e^*})_r (m/{e^*})_{r} & \text{ if } e^* \text{ divides } m,\\
1 & \text{ if }  e^* \text{ does not divide } m.\end{array}\right.$$

$$\prod_{i=1}^{m}(k^i-1)_r=(k^e-1)_r^{[m/e]}([m/e]!)_r$$

$$\prod_{i=1}^{m}(k^i-(-1)^i)_r=(k^{e^*}-(-1)^{e^*})_r^{[m/e^*]}([m/e^*]!)_r$$

\end{lem}

In Lemma~\ref{L_Epi}, we collect some known facts about $\pi$-Hall subgroups
in finite groups.

\begin{lem}\label{L_Epi}
Let $G$ be a finite group, $A$~a normal subgroup of $G$.

\begin{enumerate}[{\em(a)}]
\item If $H$ is a $\pi$-Hall subgroup of $G$ then $H\cap A$ is a $\pi$-Hall
    subgroup of $A$ and $HA/A$~is a $\pi$-Hall subgroup of $G/A$. In
    particular, a normal subgroup and a homomorphic image of an
    $\E_\pi$-group satisfy $\E_\pi$. {\em (see
    \cite[Lemma~1]{Hall1956})}\label{Epi}
\item If $M/A$ is a $\pi$-subgroup of $G/A$, then there exists a
    $\pi$-subgroup $H$ of $G$ with $M=HA$. {\em (see
    \cite[Lemma~2.1]{Gross1986})} \label{existpi}
\item An extension of a $C_\pi$-group by a $C_\pi$-group satisfies $C_\pi$.
    {\em (see  \cite[C1 and~C2]{Hall1956} or
    \cite[Proposition~5.1]{VdovinRevin2011})}
\item If $2\notin \pi$ then $\E_\pi=\C_\pi$. In particular, if $2\notin
    \pi$ then a group $G$ satisfies~$\E_\pi$ if and only if each
    composition factor of $G$ satisfies $\E_\pi$. {\em  (see
    \cite[Theorem~A]{Gross1987}, \cite[Theorem 2.3]{Gross1995},
    \cite[Theorem~5.4]{VdovinRevin2011})} \label{epi=cpi}
\item If $G$ possesses a nilpotent $\pi$-Hall subgroup then $G$ satisfies
    $\D_\pi$. {\em (see  \cite{Wielandt1954},
    \cite[Theorem~6.2]{VdovinRevin2011}) } \label{dpi_nil}
\item A group~$G$ satisfies $\D_\pi$ if and only if  both $A$ and $G/A$
    satisfy $\D_\pi$. Equivalently,  $G\in \D_\pi$ if and only if each
    composition factor of $G$ satisfies~$\D_\pi$. {\em (see
    \cite[Theorem~7.7]{RevinVdovin2006},
    \cite[Collorary~6.7]{VdovinRevin2011})} \label{Dpi_comp_crit}
\end{enumerate}
\end{lem}

\begin{lem} \label{Dpi}
{\em (\cite[Theorem~3]{Revin2008e}, \cite[Theorem~6.9]{VdovinRevin2011})} Let
$S$ be a simple group of Lie type with the base field  $\F_q$ of
characteristic $p$. Suppose $2\not \in \pi$ and $|\pi\cap\pi(S)|\ge 2$.  Then
$S$ satisfies~$\D_\pi$ if and only if the pair~$(S, \pi)$ satisfies one of
the Condition  I-IV below.
\end{lem}

\noindent{\bf Condition I.} Let  $p\in\pi$ and
$\tau=(\pi\cap\pi(S))\setminus\{p\}$. We say that~$(S,\pi)$ {\it satisfies
Condition}~I if $\tau\subseteq\pi(q-1)$ and every number from $\pi$ does not
divide~$|W(S)|$.

\smallskip

\noindent{\bf Condition II.} Suppose that $S$ is not isomorphic to
${^2B}_2(q),{^2G}_2(q),{^2F}_4(q)'$ and $p\not\in\pi$. Set
$r=\min(\pi\cap\pi(S))$ and  $\tau=(\pi\cap\pi(S))\setminus\{r\}$. Denote by
$a$ the number $e(q,r)$. We say that $(S,\pi)$ {\it satisfies Condition}~II
if there exists  $t\in\tau$ with $b=e(q,t)\ne a$ and one of the following
holds.
\begin{enumerate}[(a)]

\item  $S\simeq A_{n-1}(q)$, $a=r-1$, $b=r$, $(q^{r-1}-1)_r=r$,
    $\left[\displaystyle\frac{n}{r-1}\right]=\left[\displaystyle\frac{n}{r}\right]$
    and both $e(q,s)=b$ and $n<bs$ hold for every $s\in\tau$.

\item  $S\simeq A_{n-1}(q)$, $a=r-1$, $b=r$, $(q^{r-1}-1)_r=r$,
    $\left[\displaystyle\frac{n}{r-1}\right]=\left[\displaystyle\frac{n}{r}\right]
    +1$, $n\equiv -1 \pmod r$ and both $e(q,s)=b$ and $n<bs$ hold for every
    $s\in\tau$.

\item  $S\simeq {^2A}_{n-1}(q)$, $r\equiv 1 \pmod 4$, $a=r-1$, $b=2r$,
    $(q^{r-1}-1)_r=r$,
    $\left[\displaystyle\frac{n}{r-1}\right]=\left[\displaystyle\frac{n}{r}\right]$
    and $e(q,s)=b$ for every $s\in\tau$.

\item  $S\simeq {^2A}_{n-1}(q)$, $r\equiv 3 \pmod 4$,
    $a=\displaystyle\frac{r-1}{2}$, $b=2r$, $(q^{r-1}-1)_r=r$,
    $\left[\displaystyle\frac{n}{r-1}\right]=\left[\displaystyle\frac{n}{r}\right]$
    and $e(q,s)=b$ for every $s\in\tau$.

\item  $S\simeq {^2A}_{n-1}(q)$, $r\equiv 1 \pmod 4$, $a=r-1$, $b=2r$,
    $(q^{r-1}-1)_r=r$,
    $\left[\displaystyle\frac{n}{r-1}\right]=\left[\displaystyle\frac{n}{r}\right]
    +1$, $n\equiv -1 \pmod r$ and $e(q,s)=b$ for every $s\in\tau$.

\item $S\simeq {^2A}_{n-1}(q)$, $r\equiv 3 \pmod 4$,
    $a=\displaystyle\frac{r-1}{2}$, $b=2r$, $(q^{r-1}-1)_r=r$,
    $\left[\displaystyle\frac{n}{r-1}\right]=\left[\displaystyle\frac{n}{r}\right]
    +1$, $n\equiv -1 \pmod r$ and $e(q,s)=b$ for every $s\in\tau$.

\item $S\simeq {^2D}_n(q)$, $a\equiv 1 \pmod 2$, $n=b=2a$ and for every
    $s\in\tau$ either $e(q,s)=a$ or $e(q,s)=b$. \label{CondII_2D_n_1}

\item $S\simeq {^2D}_n(q)$, $b\equiv 1 \pmod 2$, $n=a=2b$ and for every
    $s\in\tau$ either $e(q,s)=a$ or $e(q,s)=b$. \label{CondII_2D_n_2}

\end{enumerate}

In cases (\ref{CondII_2D_n_1})-(\ref{CondII_2D_n_2}), a $\pi$-Hall subgroup
of $S\simeq {}^2D_n(q) $ is cyclic.

\smallskip

\noindent{\bf Condition III.} Suppose that $S$ is not isomorphic to
${^2B}_2(q),{^2G}_2(q),{^2F}_4(q)'$ and $p\not\in\pi$. Set
$r=\min(\pi\cap\pi(S))$ and  $\tau=(\pi\cap\pi(S))\setminus\{r\}$. Denote by
$c$ the number $e(q,r)$. We say that $(S,\pi)$ {\it satisfies Condition}~III
if ${e(q,t)=c}$ for every $t\in\tau$ and one of the following holds.
\begin{enumerate}[(a)]
\item  $S\simeq A_{n-1}(q)$ and $n<ct$ for every $t\in\tau$.

\item  $S\simeq{^2A}_{n-1}(q)$, $c\equiv 0 \pmod 4$ and $n<ct$ for every
    $t\in\tau$.

\item  $S\simeq{^2A}_{n-1}(q)$, $c\equiv 2 \pmod 4$ and $2n<ct$ for every
    $t\in\tau$.

\item  $S\simeq{^2A}_{n-1}(q)$, $c\equiv 1 \pmod 2$ and $n<2ct$ for every
    $t\in\tau$.

\item  $S$ is isomorphic to one of the groups $B_n(q)$, $C_n(q)$ or
    ${^2D}_n(q)$,  $c$ is even and $2n<ct$ for every $t\in\tau$.

\item   $S$ is isomorphic to one of the groups $B_n(q)$, $C_n(q)$ or
    $D_n(q)$, $c$ is odd and $n<ct$ for every $t\in\tau$.

\item  $S\simeq D_n(q)$,  $c$ is even and $2n\le ct$ for every $t\in\tau$.

\item  $S\simeq {^2D}_n(q)$,  $c$  is odd and $n\le ct$ for every
    $t\in\tau$.

\item  $S\simeq {^3D}_4(q)$.

\item  $S\simeq E_6(q)$  and  if $r=3$ and $c=1$ then $5,13\not\in\tau$.

\item  $S\simeq {^2E}_6(q)$ and if $r=3$ and  $c=2$ then $5,13\not\in\tau$.

\item  $S\simeq E_7(q)$; if $r=3$ and $c\in\{1,2\}$ then
    $5,7,13\not\in\tau$, and if $r=5$ and $c\in\{1,2\}$ then
    $7\not\in\tau$.

\item  $S\simeq E_8(q)$; if $r=3$ and $c\in\{1,2\}$ then
    $5,7,13\not\in\tau$, and if $r=5$ and $c\in\{1,2\}$ then
    $7,31\not\in\tau$.

\item  $S\simeq G_2(q)$.

\item  $S\simeq F_4(q)$ and if $r=3$ and $c=1$ then $13\not\in\tau$.

\end{enumerate}

\smallskip

\noindent{\bf Condition IV.} We say that $(S,\pi)$ {\it satisfies
Condition}~IV if one of the following holds.

\begin{enumerate}[(a)]

\item $S\simeq {^2B}_2(2^{2m+1})$, $\pi\cap\pi(G)$ is contained in one of
    the sets $\pi(2^{2m+1}-1)$, $\pi(2^{2m+1}\pm 2^{m+1}+1)$.

\item $S\simeq {^2G}_2(3^{2m+1})$, $\pi\cap\pi(G)$ is contained in one of
    the sets $\pi(3^{2m+1}-1)\setminus\{2\}$, $\pi(3^{2m+1}\pm
    3^{m+1}+1)\setminus\{2\}$.

\item $S\simeq {^2F}_4(2^{2m+1})'$, $\pi\cap\pi(G)$ is contained in one of
    the sets $\pi(2^{2(2m+1)}\pm 1)$, $\pi(2^{2m+1}\pm 2^{m+1}+1)$,
    $\pi(2^{2(2m+1)}\pm 2^{3m+2}\mp2^{m+1}-1)$, $\pi(2^{2(2m+1)}\pm
    2^{3m+2}+2^{2m+1}\pm2^{m+1}-1)$.

\end{enumerate}

In the next three lemmas, we recall some preliminary results about
$\U_\pi$-pro\-per\-ty.

\begin{lem}\label{pinpi}{\em \cite[Theorem 4]{VdovinManRevin2012}}
If $G\in\D_\pi$~is either an alternating group, or a sporadic simple group,
or a simple group of Lie type in characteristic $p \in \pi$, then $G$
satisfies $\U_\pi$.
\end{lem}

\begin{lem}\label{max}{\em \cite[Lemma~3]{Manzaeva2014}}
The following statements are equivalent.
\begin{enumerate}[{\em(a)}]
\item $\D_\pi=\U_\pi$.
\item In every simple $\D_\pi$-group $G$, all maximal subgroups containing
    a $\pi$-Hall subgroup of $G$ satisfy $\D_\pi$.
\end{enumerate}
\end{lem}

\begin{lem}\label{2inpi}{\em \cite[Theorem 1]{Manzaeva2014}}
If $2\in \pi$ then  $\D_\pi= \U_\pi$.
\end{lem}

 In view of Lemma~\ref{2inpi}, we consider the case when $\pi$ is a set of odd
primes.

\begin{lem} \label{t.2.1}  {\em \cite[Theorem 1]{VdovinRevin2002}}
  Let $G$~ be a group of Lie type in characteristic~$p$. Suppose that
$2,p\notin\pi$ and $H$~is a~$\pi$-Hall subgroup of~$G$. Set $r=
\min(\pi\cap\pi(G))$ and  $\tau=(\pi\cap \pi(G))\setminus\{r\}$. Then $H$ has
a normal abelian  $\tau$-Hall subgroup.
\end{lem}

In spite of Lemma \ref{t.2.1}, we say that a pair $(G,\pi)$ {\it satisfies}
\eqref{*} if

\begin{equation*}\label{*}
\tag{$*$}
  \begin{split}
   \text{every }\pi\text{-subgroup of }G\text{ has a n}&\text{ormal abelian }\tau\text{-Hall subgroup,}\\ \text{ where  }r=\min(\pi\cap\pi(G))&\text{ and
}\tau=(\pi\cap\pi(G))\setminus \{r\}.
 \end{split}
 \end{equation*}

Suppose that $G\in \E_\pi$ is a group of Lie type in characteristic~$p$ and
$2,p \notin \pi$. If $G \in \D_\pi$ then every $\pi$-subgroup of $G$ is
contained in a $\pi$-Hall subgroup of $G$, and hence by Lemma~\ref{t.2.1} we
have that $(G,\pi)$ satisfies \eqref{*}.  The following lemma gives
sufficient conditions for the validity of the converse statement.

\begin{lem}  \label{dpi_not2}{\em \cite[Theorem 5]{VdovinRevin2002}}
Let $G$~be a group of Lie type with the base field $\F_q$ of characteristic
$p$, and $G$ is not isomorphic to ${^2B}_2(q),{^2G}_2(q),{^2F}_4(q)'$.
Suppose that $2,p \not \in \pi$ and $G\in \E_\pi$. Assume further that
$e(q,t)=e(q,s)$ for every $t,s\in\pi\cap\pi(G)$. Then $G\in \D_\pi$ if and
only if $(G,\pi)$ satisfies \eqref{*}.
\end{lem}

The next lemma says, when a simple group $S$ satisfies $\E_\pi$ and does not
satisfy~$\D_\pi$.

\begin{lem}\label{epi-dpi_lie} {\em (\cite[Theorem~1.1]{Gross1986},
\cite[Theorem~6.14]{Gross1986}, \cite[Lemmas~5-7]{Revin2009around})} Let $S$
be a simple group. Suppose that $2\notin \pi$  and $S \in \E_\pi\setminus
\D_\pi$. Set $r=\min(\pi\cap \pi(S))$ and $\tau=(\pi\cap
\pi(S))\setminus\{r\}$. Then one of the following holds.

\begin{enumerate}[{\em(I)}]
\item $S\simeq O'N$ and $\pi\cap\pi(S)=\{3,5\}$.
\item $S$~is a group of Lie type with the base field $\F_q$ of
    characteristic $p$ and either {\em(A)} or {\em(B)} below is true:
    \begin{enumerate}[{\em(A)}]
    \item $p\in \pi$, $p$ divides $|W(S)|$, every $t\in(\pi\cap
        \pi(S))\setminus \{p\}$ divides $q-1$ and does not divide
        $|W(S)|$.
    \item $p\notin \pi$ and one of {\em(a)-(i)} below holds.
        \begin{enumerate}[{\em(a)}]

        \item $S\simeq \PSL_n(q)$, $e(q,r)=r-1$, $(q^{r-1}-1)_r=r$,
            $\left[\frac{n}{r-1}\right]=\left[\frac{n}{r}\right]$ and
            for every
$t\in \tau$ we have $e(q,t)=1$ and $n<t$. \label{L_epi-dpi_2B_a}
        \item $S\simeq \PSU_n(q)$, $r\equiv1\,(\mathrm{mod}\, 4)$,
            $e(q,r)=r-1$, $(q^{r-1}-1)_r=r$,
            $\left[\frac{n}{r-1}\right]={\left[\frac{n}{r}\right]}$ and
            for every $t\in \tau$ we have $e(q,t)=2$ and $n<t$.
            \label{L_epi-dpi_2B_b}
        \item $S\simeq \PSU_n(q)$, $r\equiv3\,(\mathrm{mod}\, 4)$,
            $e(q,r)=\frac{r-1}{2}$, $(q^{r-1}-1)_r=r$,
            $\left[\frac{n}{r-1}\right]={\left[\frac{n}{r}\right]}$ and
            for every $t\in \tau$ we have $e(q,t)=2$ and $n<t$.
            \label{L_epi-dpi_2B_c}
        \item $S\simeq E_6(q)$, $\pi\cap\pi(S)\subseteq \pi(q-1)$, $3,
            13 \in \pi\cap\pi(S)$, $5\notin
            \pi\cap\pi(S)$.\label{L_epi-dpi_2B_d}
        \item $S\simeq {}^2E_6(q)$, $\pi\cap\pi(S)\subseteq \pi(q+1)$,
            $3, 13 \in \pi\cap\pi(S)$, $5\notin \pi\cap\pi(S)$.
        \item $S\simeq E_7(q)$, $\pi\cap\pi(S)$ is contained in one of
            the sets $\pi(q-1)$ or $\pi(q+1)$, $3, 13 \in
            \pi\cap\pi(S)$, $5,7 \notin \pi\cap\pi(S)$.
        \item $S\simeq E_8(q)$, $\pi\cap\pi(S)$ is contained in one of
            the sets $\pi(q-1)$ or $\pi(q+1)$, $3, 13 \in
            \pi\cap\pi(S)$, $5,7 \notin \pi\cap\pi(S)$.
        \item $S\simeq E_8(q)$, $\pi\cap\pi(S)$ is contained in one of
            the sets $\pi(q-1)$ or $\pi(q+1)$, $5, 31 \in
            \pi\cap\pi(S)$, $3,7 \notin \pi\cap\pi(S)$.
        \item $S\simeq F_4(q)$, $\pi\cap\pi(S)$ is contained in one of
            the sets $\pi(q-1)$ or $\pi(q+1)$, $3, 13 \in
            \pi\cap\pi(S)$. \label{L_epi-dpi_2B_i}
        \end{enumerate}
    \end{enumerate}
\end{enumerate}
\end{lem}

\noindent{\bf Remark on Lemma~\ref{epi-dpi_lie}}. Consider simple classical
groups in characterictic~$p$. Suppose that $2,p\notin \pi$. If a simple
classical group satisfies $\E_\pi$ and
 does not satisfy $\D_\pi$,
then it must be linear or unitary by Lemma~\ref{epi-dpi_lie}. Thus if $S\in
\E_\pi$ is a  simple orthogonal or sympletic group, then $S\in \D_\pi$.
However, there are isomorphisms amongst the classical groups, and it may
happen that  a simple orthogonal or sympletic group $S\in \E_\pi$ is
isomorphic to a linear or unitary group $S_1$ (see \cite[Proposition
2.9.1]{KleiLie}). One can check in this case that $S_1\in \D_\pi$. For
instance, suppose that $S=\P\Omega^\pm_6(q)\simeq \PSL^\pm_4(q)=S_1$. Assume
also that $S_1\in \E_\pi\setminus \D_\pi$, and therefore $S_1$ satisfies one
of items \ref{L_epi-dpi_2B_a}-\ref{L_epi-dpi_2B_c} of
Lemma~\ref{epi-dpi_lie}. If $r=\min(\pi\cap\pi(S_1))$, we have that $r\le
n=4$ (see Lemma~\ref{nn} below) and so $r=3$. Then
$\left[\frac{n}{r-1}\right]\neq\left[\frac{n}{r}\right]$ and none of
conditions \ref{L_epi-dpi_2B_a}-\ref{L_epi-dpi_2B_c} holds. Hence we conclude
that $S_1=\PSL_4^\pm(q)\in \D_\pi$. Thus if $S$ is a simple orthogonal or
sympletic group  in characteristic~$p$,
 $\pi$ is a set of primes with $2,p\notin \pi$ and $S\in \E_\pi$, then $S \in
\D_\pi$.

We consider $\GL_{n}^-(q)$ as the set $\{(a_{ij})\in \GL_{n}(q^2)\mid
(a_{ij}^{q}) = ((a_{ij})^{-1})^{\top}\}$, where $(a_{ij})^{\top}=(a_{ji})$ is
the transposed of $(a_{ij})$. In the following statement, we specify the
structure of $\pi$-Hall subgroup in $\GL_n^\eta(q)$ in   case $2,p\notin \pi$
and $\GL_n^\eta(q)$ does not satisfy $\D_\pi$.

\begin{lem}\label{hall-gl}
Let $G=\GL_n^\eta(q)$, where $q=p^{m}$ and $p$ is a prime. Denote by $D$  the
subgroup of all diagonal matrices in $G$, so that  $D\simeq (q-\eta)^{n}$,
and by $P$ the subgroup of permutation matrices of $G$,  so that $P\simeq
S_{n}$ and $P$ normalizes~$D$. Suppose that $2,p\notin \pi$  and $G \in
\E_\pi\setminus \D_\pi$. Set $r=\min(\pi\cap \pi(G))$ and $\tau=(\pi\cap
\pi(G))\setminus\{r\}$. Then the following statements hold.
\begin{enumerate}[{\em(a)}]
 \item $r$ does not divide $q-\eta$, $\tau\subseteq\pi(q-\eta)$, and a
     $\pi$-Hall subgroup $T$ of $D$ is isomorphic to $(q-\eta)_\tau^n$.
 \item $P$ is a $\tau'$-group, a $\pi$-Hall subgroup of $P$ is nontrivial
     and coincides with a Sylow $r$-subgroup $R$ of $P$ and $R$ is
     elementary abelian of order $r^{[n/r]}$.
  \item $R$ normalizes $T$ and $TR$ is a $\pi$-Hall subgroup of~$G$. In
      particular, $|G|_\pi=(q-\eta)^n_\tau r^{[n/r]}$.
  \item  Consider the automorphism $\varphi:(a_{ij})\mapsto (a_{ij}^{p})$
      of $G$. Then $\varphi$ normalizes~$T$ and centralizes $R$.
  In particular, $\varphi$ normalizes $TR$.
  \item  Let $d=\left[\frac{n}{r}\right]$ and $n=dr+k$. Then
      $C_{TR}(R)\simeq (q-\eta)_\tau^{d+k}\times R$ and
      $m_t(C_{TR}(R))=d+k$ for every $t\in\tau$.
\end{enumerate}
\end{lem}

\begin{proof}
 Since $G\in \E_\pi\setminus\D_\pi$ and $2\notin \pi$, Lemma
\ref{L_Epi}(\ref{epi=cpi}, \ref{Dpi_comp_crit}) implies that
 the unique nonabelian composition factor  $S=\PSL_n^\eta(q)$ of $G$ lies in
$\E_\pi\setminus\D_\pi$. Since $p\notin \pi$,
 we obtain that $S$ satisfies one of
items~\ref{L_epi-dpi_2B_a}-\ref{L_epi-dpi_2B_c} of Lemma~\ref{epi-dpi_lie}.
 Observe that $\pi(G)=\pi(S)$. Denote $e(q,r)$ by~$e$.

(a) Since $D\simeq (q-\eta)^n$ is abelian, $D$ satisfies $\D_\pi$.
Items~\ref{L_epi-dpi_2B_a}-\ref{L_epi-dpi_2B_c} of Lemma~\ref{epi-dpi_lie}
imply $\tau\subseteq\pi(q-\eta)$. It suffies, therefore, to prove that
  $r$ does not divide $q-\eta$.

  If $\eta=+$, then $S=\PSL_n(q)$ satisfies item
\ref{L_epi-dpi_2B_a}
  of Lemma~\ref{epi-dpi_lie}. It now follows that $e=r-1>1$, and so $r$ does not
divide $q-1$.

If $\eta=-$, then $S=\PSU_n(q)$ satisfies
  either~\ref{L_epi-dpi_2B_b} or \ref{L_epi-dpi_2B_c} of
Lemma~\ref{epi-dpi_lie}. It is easy to see  that in both cases $e$~is not
equal to 2,
  and so $r$ does not divide $q+1$. Observe also that in both cases $e^*=r-1$.
Thus $r\notin \pi(q-\eta)$, as required.

(b) It follows from items~\ref{L_epi-dpi_2B_a}-\ref{L_epi-dpi_2B_c} of
Lemma~\ref{epi-dpi_lie} that $n<t$ for every $t\in \tau$,
  and hence $P\simeq S_n$ is a $\tau'$-group. Then $|\pi\cap\pi(P)|\le 1$, and
$P$ satisfies $\D_\pi$ by the Sylow theorems.
  As we have seen above,  if $\eta=+$ then $e=r-1$, and if $\eta=-$ then
$e^*=r-1$.  Lemma~\ref{r4ast'} implies that
\begin{equation*}
\begin{split}
 |\GL_n(q)|_r=(q^e-1)_r^{\left[n/e\right]}\left(\left[n/e\right]!\right)_r,\\
|\GU_n(q)|_r=(q^{e^*}-(-1)^{e^*})_r^{\left[n/e^*\right]}\left(\left[n/e^*\right]!\right)_r.
\end{split}
\end{equation*}
Hence
\begin{equation}\label{OrdersOfGLGU}
 \vert \GL_n^\eta(q)\vert_r=(q^{r-1}-1)_r^{\left[n/(r-1)\right]} \left(\left[n/(r-1)\right]!\right)_r.
\end{equation}

 Since $r$ divides the order
  of $G$, we have that $n\ge r-1$. Also, since
$\left[\frac{n}{r-1}\right]=\left[\frac{n}{r}\right]$,
  we have that $n\ge r$, and so a Sylow $r$-subgroup $R$ of $P$ is nontrivial. Then $d=\left[\frac{n}{r}\right]>0$ and it follows from
$d=\left[\frac{n}{r-1}\right]$ that
\begin{equation}\label{n}
   n=dr+k=d(r-1)+(d+k) \text{ and } 0< d+k< r-1.
   \end{equation}
 In particular, $d< r-1$ and a Sylow $r$-subgroup $R$ of $P$ is isomorphic to
$r^d$ by \cite[11.3.1, Example~III]{KargMerz}. Moreover
\begin{equation*}
n=d(r-1)+(d+k)\leq d(r-1)+r-2\leq r(r-2).
\end{equation*}

  \medskip

(c)  Since $T$ is a characteristic
  subgroup of $D$, we conclude that $R$ normalizes $T$, and so $TR$ is a
$\pi$-subgroup of $G$. To prove that $TR$ is a $\pi$-Hall
  subgroup of~$G$,  it suffices to show that $|G|_\pi=|TR|=(q-\eta)^n_\tau
r^{[n/r]}$.

 Items~\ref{L_epi-dpi_2B_a}-\ref{L_epi-dpi_2B_c}
  of Lemma~\ref{epi-dpi_lie} yield $(q^{r-1}-1)_r=r$ and we have seen that
$\left[{n}/{(r-1)}\right]=d<r-1$. Thus,
  $$\left(\left[{n}/{(r-1)}\right]!\right)_r=1.$$
  In view of~(\ref{OrdersOfGLGU}),
  we conclude that $|G|_r=r^{[n/r-1]}=r^{[n/r]}$, as required.

   Calculate $t$-part of order of $G$ for every $t\in\tau$. It follows from
items~\ref{L_epi-dpi_2B_a}-\ref{L_epi-dpi_2B_c}
  of Lemma~\ref{epi-dpi_lie} that $e(q,t)=e(q,s)$ for every $t,s \in \tau$.
Denote $e(q,t)$ by~$f$. Observe that if $\eta=+$
  then $f=1$, and if $\eta=-$ then $f=2$ and $f^*=1$. Since $n<t$, Lemma~\ref{r4ast'}
yields

$$|\GL_n(q)|_t=(q^f-1)_t^{\left[{n}/{f}\right]}\left(\left[{n}/{f}\right]
!\right)_t=(q-1)^n_t(n!)_t=(q-1)^n_t$$


$$|\GU_n(q)|_t=(q^{f^*}-(-1)^{f^*})_t^{\left[n/f^*\right]}\left(\left[
n/f^*\right]!\right)_t=(q+1)^n_t(n!)_t=(q+1)^n_t$$
  Thus $|G|_\tau=(q-\eta)^n_\tau$, as required.

(d) Clearly, $\varphi$ normalizes $D$ and centralizes $P$. In particular,
$\varphi$ also centralizes~$R$. Since $T$ is a characteristic
  subgroup of $D$, we have that $\varphi$ normalizes~$T$.

(e) We will denote by $D_i$ the subgroup of $D$ consisting of all matrices
$\diag(1,\dots,\alpha,\dots, 1)$ where $\alpha$ is an element of the
corresponding field such that $\alpha^q=\alpha$ and $\alpha$ is placed on the
$i$th position. Let $T_i$ be the unique $\tau$-Hall subgroup of $D_i$. It is
clear that both $P$ and $R$ act on the set
$$\Omega=\{T_1,\dots, T_n\}$$
via conjugation. Since $R$ is a Sylow $r$-subgroup of $P\simeq S_n$, and in
view of~(\ref{n}),  $R$ has $d$ orbits of size $r$ and $k$ orbits of size 1
on~$\Omega$ (see  \cite[11.3.1, Example~III]{KargMerz}). It is easy to see
that if $\Delta$ is an orbit of $R$ on $\Omega$ then
$$
C_{\langle\Delta\rangle}(R)\simeq (q-\eta)_\tau \text{ and }$$
$$C_T(R)=\langle C_{\langle\Delta\rangle}(R)\mid \Delta \text{ is an orbit of }
R \text{ on }\Omega\rangle\simeq
(q-\eta)^{d+k}_\tau.
$$
  Since $R$ is abelian, we obtain that $$C_{TR}(R)\simeq
(q-\eta)_\tau^{d+k}\times R.$$ Thus, for every $t\in \tau$, the maximal rank
of elementary
  abelian $t$-subgroup $m_t(C_{TR}(R))$ is equal to $d+k$.
  \end{proof}

Since $\PSL_n^\eta(q)$ is a unique nonabelian composition factor of
$\GL_n^\eta(q)$, as a consequence of  (a) and the proof of  (b)  we obtain

\begin{lem}\label{nn}
Let $S=\PSL_n^\eta(q)$, where $q=p^m$, and let  $2,p\not\in\pi$. Set
$r=\min(\pi\cap\pi(S))$. Then $S\in \E_\pi\setminus \D_\pi$ implies
  $\gcd(n,q-\eta)_\pi=1$ and $r\le n\le r(r-2)$.
 \end{lem}

\begin{lem}\label{elab-gl}
Let $G=\GL_n^\eta(q)$, where $q=p^{m}$ and $p$ is a prime. Consider the
automorphism $\varphi:(a_{ij})\mapsto (a_{ij}^{p})$ of $G$. Suppose that
$2,p\notin \pi$  and $G \in \E_\pi\setminus \D_\pi$. Set $r=\min(\pi\cap
\pi(G))$ and $\tau=(\pi\cap \pi(G))\setminus\{r\}$.  Take arbitrary
$t\in\tau$. Then $G$ contains an $r$-subgroup $R$ and an elementary abelian
$t$-subgroup $K$ such that
\begin{enumerate}[{\em(a)}]
\item $R$ is a Sylow $r$-subgroup of $G$ centralized by $ \varphi_{2'}$;
\item $K$ is $\varphi$-invariant, $K\le C_G(R)$, and the rank of $K$ equals
    $2d+k$, where $d$ and $k$ are defined by $d=[n/r]$ and
    $k=n-dr$;\label{elab-gl_b}
\item if $m$ is divisible by $t$  and $\psi\in \langle\varphi\rangle$ is of
    order~$t$, then $K$ in {\rm(\ref{elab-gl_b})} can be chosen such that
    $K\le C_G(\psi)\cap C_G(R).$\label{elab-gl_c}
\end{enumerate}
\end{lem}

\begin{proof} It follows from Lemma~\ref{hall-gl}(b) that $d=[n/r]>0$.
The equality  $n=d(r-1)+d+k$ implies that  $G$ contains a subgroup of
block-diagonal matrices
$$X=\underbrace{\GL_{r-1}^\eta(q)\times\ldots\times \GL_{r-1}^\eta(q)}_{d \text{
times}}\times \underbrace{\GL_1^\eta(q)\times\ldots\times\GL_1^\eta(q)}_{k+d
\text{ times}}.$$
This subgroup is $\varphi$-invariant. Every $\GL_{r-1}^\eta(q)$ contains a
subgroup $\GL_{r-1}^\eta(p)$  centralized by~$\varphi_{2'}$. Consider a
subgroup
$$Y=\underbrace{\GL_{r-1}^\eta(p)\times\ldots\times \GL_{r-1}^\eta(p)}_{d \text{
times}}\times
\underbrace{\GL_1^\eta(p)\times\ldots\times\GL_1^\eta(p)}_{k+d \text{ times}}$$
of~$X$. It follows from Fermat's little theorem that
$|\GL_{r-1}^\eta(p)|_r>1$. So, $Y$ contains an elementary abelian
$r$-subgroup $R$ of order $r^d$ and $R$ is centralized by $\varphi_{2'}$.
 By Lemma~\ref{hall-gl} we have that  $|G|_r=r^d$ and $R$ is a Sylow
$r$-subgroup of~$G$, as required.
  Observe that $Z(X)\simeq (q-\eta)^{2d+k}$. Lemma~\ref{hall-gl}(a) implies that
$t$ divides $q-\eta$. Thus the unique maximal elementary abelian $t$-subgroup
$K$ of $Z(X)$ is a desired subgroup.

If  $m$ is divisible by $t$, then $q=q_0^t$ where $q_0=p^{m/t}$. By Fermat's
little theorem $ q_0\equiv q\pmod t$
 and since
$q\equiv \eta \pmod t $ we obtain that
$$q_0\equiv\eta\pmod t.$$ Consider a subgroup $X_0$ such that $Y\le X_0\le X$
and
$$
X_0=\underbrace{\GL_{r-1}^\eta(q_0)\times\ldots\times \GL_{r-1}^\eta(q_0)}_{d
\text{ times}}\times \underbrace{\GL_1^\eta(q_0)\times\ldots
\times\GL_1^\eta(q_0)}_{k+d \text{ times}}.
$$
Clearly, $X_0\le\GL_n^\eta(q_0)= C_G(\psi)$. It is easy to see that
$Z(X_0)\le Z(X)$ and since $t$ divides $q_0-\eta$, we conclude that $K\le
Z(X_0)\le C_G(\psi)$.
\end{proof}

We also need some information about automorphisms of groups of Lie type. Let
$S$~be a simple group of Lie type. Definitions of diagonal, field and graph
automorphisms of $S$ agree with that of \cite{Steinberg}. The group of
inner-diagonal automorphisms of $S$ is denoted by  $\widehat{S}$. By
\cite[3.2]{Steinberg}, there exists a field automorphism $\rho$ of~$S$ such
that every automorphism $\sigma$ of $S$ can be written
$\sigma=\beta\rho^l\gamma$, with $\beta$ and  $\gamma$ being an
inner-diagonal and a graph automorphisms, respectively, and $l\ge 0$. The
group $\langle \rho\rangle$ is denoted by $\Phi_S$. In view of
\cite[7-2]{GorLyons1983}, the group $\Phi_S$ is determined up to
$\widehat{S}$-conjugacy.  Since $S$ is centerless, we can identify $S$ with
the group of its inner automorphisms.

\begin{lem}{\em  \cite[3.3, 3.4, 3.6]{Steinberg}}\label{automorphisms}
Let $S$~be a simple group of Lie type over $\F_q$ of characteristic $p$. Set
$A=\Aut(S)$ and $\widehat{A}=\widehat{S}\Phi_S$. Then the following
statements hold.
\begin{enumerate}[{\em(a)}]
\item $S\le \widehat{S} \le \widehat{A}\le A$ is a normal series for $A$.
\item  $\widehat{S}/S$ is abelian; $\widehat{S}=S$ for the groups $E_8(q),
    F_4(q), G_2(q),{}^3D_4(q)$, in other cases the order of $\widehat{S}/S$
    is specified in {\em Table~\ref{D}}.
    \begin{longtable}{cc}
\caption{} \label{D} \\ \hline $S$ & $|\widehat{S}/S|$ \\ \hline
        $A_l(q)$ & $\gcd(l+1,q-1)$ \\
        ${}^2A_l(q)$ & $\gcd(l+1,q+1)$ \\
        $B_l(q)$, $C_l(q)$, $E_7(q)$ & $\gcd(2,q-1)$ \\
        $D_l(q)$ & $\gcd(4,q^l-1)$ \\
        ${}^2D_l(q)$ & $\gcd(4,q^l+1)$ \\
        $E_6(q)$ & $\gcd(3,q-1)$ \\
        ${}^2E_6(q)$ & $\gcd(3,q+1)$ \\ \hline
\end{longtable}
\item $A=\widehat{A}$ with the exceptions: $A/\widehat{A}$ has order $2$ if
    $S$ is $A_l(q)$ ($l\ge 2$),  $D_l(q)$  ($l\ge 5$) or $E_6(q)$, or if
    $S$ is $B_2(q)$ or $F_4(q)$ and $q=2^{2n+1}$, or if  $S$ is $G_2(q)$
    and $q=3^{2n+1}$; $A/\widehat{A}$ is isomorphic to $S_3$ if $S$ is
    $D_4(q)$.
\end{enumerate}
\end{lem}

Below we need an information about maximal subgroups of groups of Lie type.
For classical groups we use the Aschbacher theorem
\cite[Theorem~1.2.1]{KleiLie} and for the information about subgroups lying
in the Aschbacher classes we refer to \cite{KleiLie}. Maximal subgroups of
exceptional groups of Lie type are specified in Lemmas \ref{max_excep} and
\ref{max_2F4}.

Let $G$ be a finite exceptional simple group of Lie type over $\F_q$, where
$q=p^a$.  Then by \cite{Steinberg1968} there is a simple adjoint algebraic
group $\overline{G}$ over the algebraic closure of $\F_q$\,, and a surjective
endomorphism $\sigma$  of $\overline{G}$ such that
$G=O^{p'}(\overline{G}_\sigma)$, the subgroup of $\overline{G}_\sigma$
generated by all its $p$-elements.

\begin{lem}\label{max_excep} {\em \cite[Theorem 2]{LieSeitz}}
Let $G=O^{p'}(\overline{G}_\sigma)$~be a finite exceptional group of Lie
type, $G_1$  is chosen so that $G\le G_1 \le \Aut(G)$, and let $M$ be a
maximal subgroup of $G_1$ such that $G\not\leq M$. Then either $F^*(M)$ is
simple, or one of the following holds.

    \begin{enumerate}[{\em(a)}]
    \item $M=N_{G_1}(D_\sigma)$, where $D$ is a $\sigma$-stable closed
        connected subgroup and $D$~is either parabolic or reductive of
        maximal rank.
    \item $M=N_{G_1}(E)$, where $E$ is an elementary abelian $s$-subgroup
        with $s$ prime and $E\le \overline{G}_\sigma$; the pair $(G, E)$ is
        as in {\em Table~\ref{Table2}}, in each case $s\neq p$.
    \item $M$ is the centralizer of a graph, field, or graph-field
        automorphism of $G$ of prime order.
    \item $\overline{G}=E_8$, $p>5$ and
        $F^*(M)\in\{{\PSL_2(5)\times\PSL_2(9)},{\PSL_2(5)\times
        \PSL_2(q)}\}$.
    \item $F^*(M)$ is as in   {\em Table \ref{Table3}}.
    \end{enumerate}

 {\small
\begin{longtable}{cccc}
\caption{ } \label{Table2} \\ \hline $G$ & $E$ &$N_{\overline{G}_\sigma}(E)$
& \mbox{ \em Conditions }\\ \hline
$G_2(p)$ & $2^3$ & $ 2^3 \arbitraryext \SL_3(2)$ & \\
${}^2G_2(3)'$ & $2^3$ & $2^3\arbitraryext 7$ & \\
$F_4(p)$ & $3^3$ & $3^3\arbitraryext \SL_3(3)$ & $p\ge 5$\\
$E_6^\eta(p)$ & $3^3$ & $3^{3+3}\arbitraryext \SL_3(3)$ &  $p \equiv \eta
\pmod{3}, p\ge 5$\\
$E_7(q)$ & $2^2$ &$(2^2\times (\P\Omega^+_8(q)\arbitraryext 2^2))
\arbitraryext
S_3$ & $\P\Omega^+_8(q)\arbitraryext 2^2= \widehat{D_4}(q)$ \\
$E_8(p)$ & $2^5$ & $2^{5+10}\arbitraryext \SL_5(2)$ &  \\
$E_8(p^a)$ & $5^3$ & $5^{3}\arbitraryext \SL_3(5)$ & $p\neq2, 5;\, a=\left\{
\begin{array}{l}
1,  \text{ if } 5 \mid p^2-1\\
2,  \mbox{ if   } 5 \mid  p^2+1 .
\end{array}
\right.$ \\
${}^2E_6(2)$ & $3^2$ & $N_G(E)=(3^2\arbitraryext \ldbrack 8 \rdbrack)\times
(\PSU_3(3)\arbitraryext 2)$ & \\
$E_7(3)$ & $2^2$ & $N_G(E)=(2^2:3)\times F_4(3)$ & \\
 \hline
\end{longtable} }
\begin{longtable}{cl}
\caption{\label{Table3} } \\ \hline
$G$ & $F^*(M)$\\ \hline
$F_4(q)$ &  $\PSL_2(q)\times G_2(q)$ $(p\ge 3, q\ge 5)$\\
$E_6^\pm(q)$ & $\PSL_3(q)\times G_2(q)$, $\PSU_3(q)\times G_2(q)$ $(q\ge 3)$\\
$E_7(q)$ &  $\PSL_2(q)\times \PSL_2(q)$ $(p \ge 5)$, $\PSL_2(q)\times G_2(q)$
$(p\ge 3, q\ge 5)$, \\
                  &  $\PSL_2(q)\times F_4(q)$ $(q \ge 4)$, $G_2(q)\times
\PSp_6(q)$\\
$E_8(q)$ &  $\PSL_2(q)\times \PSL_3^\pm(q)$ $(p\ge 5)$, $\PSL_2(q)\times
G_2(q^2)$ $(p\ge 3, q\ge 5)$, \\
 & $G_2(q)\times F_4(q)$, $\PSL_2(q)\times G_2(q)\times G_2(q)$ $(p\ge 3, q\ge
5)$ \\ \hline
\end{longtable}

\end{lem}

To simplify our proof of Theorem 1 we need a list of maximal subgroups
of~${}^2F_4(q)$.

\begin{lem}\label{max_2F4}{\em \cite[Main Theorem]{Malle1991}}
Every maximal subgroup of $G={}^2F_4(q)$, $q=2^{2m+1}$, $m\geq1$, is
isomorphic to one of the following.
\begin{enumerate}[{\em(a)}]
\item $\ld q^{11}\rd\splitext (A_1(q)\times (q-1))$.
\item $\ld q^{10}\rd\splitext ({}^2B_2(q)\times (q-1))$.
\item $\SU_3(q)\splitext 2$.
\item $((q+1)\times (q+1))\splitext \GL_2(3)$.
\item $((q-\sqrt{2q}+1)\times (q-\sqrt{2q}+1))\splitext \ld 96\rd$ if
    $q>8$.
\item $((q+\sqrt{2q}+1)\times (q+\sqrt{2q}+1))\splitext \ld 96\rd$.
\item $(q^2-\sqrt{2q}q+q-\sqrt{2q}+1)\splitext 12$.
\item $(q^2+\sqrt{2q}q+q+\sqrt{2q}+1)\splitext 12$.
\item $\PGU_3(q)\splitext 2$.
\item ${}^2B_2(q) \wr 2$.
\item $B_2(q)\splitext 2$.
\item ${}^2F_4(q_0)$, if $q_0=2^{2k+1}$ with $(2m+1)/(2k+1)$~prime.
\end{enumerate}
\end{lem}


\section{Proof of the main theorem}

In view of Lemma \ref{2inpi}, we may assume that $2\notin \pi$.  By Lemma
\ref{max} it is sufficient to prove the following statement
 \begin{equation}\label{ms}
  \begin{split}
    \text{in each simple non}& \text{abelian $\D_\pi$-group $G$, } \\
              \text{all maximal subgroups containing }& \text{a $\pi$-Hall
subgroup of $G$ satisfy~$\D_\pi$.}
 \end{split}
 \end{equation}
  Statement \eqref{ms} is true for alternating and sporadic simple groups, and
simple groups of Lie type, if the characteristic $p$ lies in $\pi$, by Lemma
\ref{pinpi}. Thus we remain to consider simple groups of Lie type in
characteristic $p$ with $p\notin \pi$. So we assume that $G$~is a simple
$\D_\pi$-group of Lie type with the base field  $\F_q$ of characteristic $p$
and
 $2, p\not \in \pi$.

 In view of the Sylow theorems, we suppose that $|\pi\cap\pi(G)|\ge 2$. So  $G$
satisfies one of Conditions~II-IV of Lemma~\ref{Dpi}.

 Throughout this section, let $H$~be a $\pi$-Hall subgroup of $G$, $M$ a maximal
subgroup of $G$ with  $H\le M$. Clearly, $H$ is a $\pi$-Hall subgroup of $M$,
in particular $M\in \E_\pi$. We proof \eqref{ms} if we show that $M$
satisfies $\D_\pi$.

  Assume by contradiction that $M$ does not satisfy $\D_\pi$.  Since $M\in
\E_\pi$, Lemma~\ref{L_Epi}(d)
  implies that every composition factor of $M$ satisfies  $\E_\pi$. Since
$M\notin \D_\pi$, Lemma~\ref{L_Epi}(\ref{Dpi_comp_crit}) implies
   that $M$ has a nonabelian composition factor $S\in \E_\pi\setminus \D_\pi$.

Recall that $r=\min({\pi\cap \pi(G)})$ and
$\tau=(\pi\cap\pi(G))\setminus\{r\}$. We proceed in a series of steps.

\bigskip
{\bf Step 1. }The following statements hold.
\begin{enumerate}[(a)]
\item $|\pi\cap\pi(S)|\ge 2 $;
\item $r\in \pi(S)$;
\item $(S,\pi)$ satisfies \eqref{*};
\item $S\simeq \PSL^\eta_{n_1}(q_1)$ for some $q_1$, $n_1$ and $\eta$; $S$
    satisfies one of items \ref{L_epi-dpi_2B_a}-\ref{L_epi-dpi_2B_c} of
    Lemma~\ref{epi-dpi_lie}. In particular, $\tau\cap\pi(S)\subseteq
    \pi(q_1-\eta)$,
 and $t>n_1$ for every $t\in \tau\cap\pi(S)$;
\item $r\notin \pi(q_1-\eta)$, $\gcd(n_1, q_1-\eta)_\pi=1$ and $r\le n_1
    \le r(r-2)$:
\item $\Out(S)$ is an $r'$-group.
\end{enumerate}

\noindent{\bf Note:} As we mentioned in Remark on Lemma \ref{epi-dpi_lie},
$S$ cannot be isomorphic to an orthogonal or sympletic group.

\bigskip

(a) If $|\pi\cap\pi(S)|\le 1$ then $S\in \D_\pi$ by the Sylow theorems.
Consequently, $|\pi\cap\pi(S)|\ge 2 $.

 (b), (c) Since $G$ satisfies~$\D_\pi$, every $\pi$-subgroup of~$G$ is contained
in a $\pi$-Hall subgroup of~$G$.
 It follows from Lemma \ref{t.2.1} that $(G, \pi)$ satisfies \eqref{*}, i.\,e.
every $\pi$-subgroup of $G$ has a
 normal abelian $\tau$-Hall subgroup. Lemma~\ref{L_Epi}(\ref{existpi}) implies
that every  $\pi$-subgroup of~$S$
 is a homomorphic image of a $\pi$-subgroup of~$M$ and hence of~$G$. Thus every
$\pi$-subgroup of~$S$ possesses
 a normal abelian $\tau$-Hall subgroup, in particular, a $\pi$-Hall subgroup of
$S$ possesses a normal abelian $\tau$-Hall subgroup.

   If $r\notin \pi(S)$ then a $\pi$-Hall subgroup of~$S$ is abelian. So by
Lemma~\ref{L_Epi}(\ref{dpi_nil}) we have
   $S\in \D_\pi$,  a contradiction.  Therefore, we conclude $r\in \pi(S)$, as
required. So $r=\min(\pi\cap\pi(S))$
   and $(\pi\cap\pi(S))\setminus \{r\}=\tau\cap \pi(S)$. Since every
$\pi$-subgroup of~$S$ possesses a normal abelian
   $\tau$-Hall subgroup, we obtain that $(S,\pi)$ satisfies \eqref{*}.

 (d) Since $2\notin \pi$ and $S\in \E_\pi\setminus \D_\pi$, the possibilities
for $S$ are determined in
 Lemma~\ref{epi-dpi_lie}.

 Suppose that $S$ satisfies item (I) of
Lemma~\ref{epi-dpi_lie}. Then $S\simeq O'N$ and
 $\pi\cap\pi(S)=\{3,5\}$.  But a $\{3,5\}$-Hall subgroup of $O'N$ does not
possess a normal Sylow 5-subgroup (see proof
  of \cite[Theorem 6.14]{Gross1986}), hence $(S,\pi)$ does not satisfy \eqref{*}
and this case is impossible.

  Consequently, $S$~is a group of Lie type with a base field $\F_{q_1}$ of a
characteristic~$p_1$.
  Assume first that $S$ satisfies item II(A) of Lemma~\ref{epi-dpi_lie}, and so
$p_1 \in \pi$. If $p_1\neq r$ then
  $r\in (\pi\cap\pi(S))\setminus\{p_1\}$ and $r$ does not divide $|W(S)|$. Since
$\pi(|W(S)|)=\pi(l!)$ for some natural~$l$,
  we obtain that $l<r<p_1$ and it contradicts the fact that $p_1$ divides
$|W(S)|$.  Suppose now that $p_1=r$.
  Denote by $U$ a Sylow $p_1$-subgroup of $S$. In view of
\cite[Theorem~3.2]{Gross1986}, a Borel subgroup $B=N_S(U)$
  contains a $\pi$-Hall subgroup $H_0$ of $S$. Since a $\tau$-Hall subgroup $Q$
of $H_0$ is normal (in $H_0$), we obtain
  that   $H_0=U\times Q$. So $H_0$ is nilpotent, and $S\in \D_\pi$ by
Lemma~\ref{L_Epi}(\ref{dpi_nil}), a contradiction.

   Hence  $S$ satisfies item II(B) of Lemma~\ref{epi-dpi_lie}, in particular,
$p_1\notin \pi$. If $S$ satisfies one of
   items  \ref{L_epi-dpi_2B_d}-\ref{L_epi-dpi_2B_i}, then
$\pi\cap\pi(S)\subseteq \pi(q_1\pm 1)$, and therefore
   $e(q_1,t)=e(q_1,s)$ for every  $t,s \in \pi\cap\pi(S)$. Recall that $2,p_1
\notin \pi$ and $S\in \E_\pi$.
   Now $(S,\pi)$ satisfies \eqref{*} by Step~1(c), so Lemma~\ref{dpi_not2}
implies that $S$ satisfies $\D_\pi$,  a contradiction.
   Thus $S$ satisfies one of items  \ref{L_epi-dpi_2B_a}-\ref{L_epi-dpi_2B_c} of
Lemma~\ref{epi-dpi_lie}, in particular, $$S\simeq \PSL_{n_1}^\eta(q_1)$$ for
some $q_1$, $n_1$ and $\eta$.

Now the rest of statement (d)  follows from items
\ref{L_epi-dpi_2B_a}-\ref{L_epi-dpi_2B_c} of Lemma~\ref{epi-dpi_lie}.

   (e) The statements follow from Lemma~\ref{nn}.

   (f) In view of (e) and Lemma~\ref{automorphisms}, it is sufficient to prove
that $|\Phi_S|$ is an $r'$-group.
   If $r$ divides $|\Phi_S|$ then $q_1=q_0^r$ for some
   $q_0$ and this equality  contradicts the conclusion $(q_1^{r-1}-1)_r=r$ in
\ref{L_epi-dpi_2B_a}-\ref{L_epi-dpi_2B_c} of Lemma~\ref{epi-dpi_lie}.

   Indeed,  suppose that  $S \simeq \PSU_{n_1}(q_1)$ and $r\equiv 1\pmod 4$ or
$S \simeq \PSL_{n_1}(q_1)$ i.\,e. $S$ satisfies
   \ref{L_epi-dpi_2B_a}-\ref{L_epi-dpi_2B_b} of Lemma~\ref{epi-dpi_lie}.
   Under these conditions $e(q_1,r)=r-1$. Since $q_1^i-1$ is divisible by
$q_0^i-1$ for every $i$, we conclude that
   $e(q_0,r)=r-1$. Now

$$q_0^{r-1}-1=\left(q_0^{\frac{r-1}{2}}-1\right)\left(q_0^{\frac{r-1}{2}}
+1\right)$$
   implies that
   $q_0^{(r-1)/2}+1$ is divisible by $r$, i.\,e.  $q_0^{(r-1)/2}\equiv -1\pmod
r$. Therefore
   $$\sum_{i=0}^{r-1}(-1)^{r-1-i}q_0^{\frac{r-1}{2}i}\equiv r\pmod r$$
   and we obtain that

   $$
q_1^{r-1}-1=\left(q_1^{\frac{r-1}{2}}-1\right)\left(q_0^{\frac{r-1}{2}r}
+1\right)=\left(q_1^{\frac{r-1}{2}}-1\right)\left(q_0^{\frac{r-1}{2}}+1\right)
   \left(\sum_{i=0}^{r-1}(-1)^{r-1-i}q_0^{\frac{r-1}{2}i}\right)
   $$
   is divisible by $r^2$; a contradiction.

    Now, suppose that  $S \simeq \PSU_{n_1}(q_1)$ and $r\equiv 3\pmod 4$, i.\,e.
$S$ satisfies
   \ref{L_epi-dpi_2B_c} of Lemma~\ref{epi-dpi_lie}. Then  $e(q_1,r)=(r-1)/2$ and
$e(q_0,r)=(r-1)/2$. This implies that
     $q_0^{(r-1)/2}\equiv 1\pmod r$ and
    $$\sum_{i=0}^{r-1}q_0^{\frac{r-1}{2}i}\equiv r\pmod r.$$
    Hence
    $$
q_1^{r-1}-1=\left(q_1^{\frac{r-1}{2}}+1\right)\left(q_0^{\frac{r-1}{2}r}
-1\right)=\left(q_1^{\frac{r-1}{2}}+1\right)\left(q_0^{\frac{r-1}{2}}-1\right)
   \left(\sum_{i=0}^{r-1}q_0^{\frac{r-1}{2}i}\right)
   $$
   is divisible by $r^2$; a contradiction again.

\bigskip

{\bf Step 2.}   $M$ is not  almost simple.

\bigskip

Assume that $M$ is an almost simple group. Therefore, $S$ is a unique
nonabelian composition factor of $M$ and we may assume that $S\le M \le
\Aut(S)$. Since $M$ contains a $\pi$-Hall subgroup  $H$ of $G$, we arrive at
a contradiction with $G\in \D_\pi$ if we find a $\pi$-subgroup of $M$ (and
hence of $G$) which is not isomorphic to any subgroup of $H$.

In order to prove Step 2, first, for every $t\in\tau\cap \pi(S)$, we estimate
$m_t(C)$, where $C$ is the centralizer in $H$ of a Sylow $r$-subgroup of $H$
(and of both $M$ and~$G$, of course) and, second, we find  an elementary
abelian $t$-subgroup $E$ of~$M$, which centralizes a Sylow $r$-subgroup $R_0$
of $H$ and whose rank is greater than $m_t(C)$. It is clear that $ER_0$ is
not isomorphic to any subgroup of~$H$. In particular, $ER_0$ is not conjugate
in $G$ to any subgroup of~$H$.

Consider the group $\GL_{n_1}^\eta(q_1)$ first. Recall that
$$\GL_{n_1}^-(q_1)=\{(a_{ij})\in
\GL_{n}(q_1^2)\mid (a_{ij}^{q_1}) = ((a_{ij})^{-1})^{\top}\},$$ where
$(a_{ij})^{\top}=(a_{ji})$ is the transposed of $(a_{ij})$. Let $\varphi$ be
an automorphism of $\GL_{n_1}^\eta(q_1)$ defined by $\varphi:(a_{ij})\mapsto
(a_{ij}^{p_1})$, where $p_1$ is the characteristic of~$\F_{q_1}$. Let $T$ be
a $\pi$-Hall subgroup of the subgroup of all diagonal matrices in
$\GL_{n_1}^\eta(q_1)$, and $R$ a Sylow $r$-subgroup of the subgroup of
permutation matrices of  $\GL_{n_1}^\eta(q_1)$. Denote by $\chi$ the
$\pi$-part of $\varphi$. Then by Lemma~\ref{hall-gl} we have that $TR$ is a
$\pi$-Hall subgroup of $\GL_{n_1}^\eta(q_1)$  and $H_1=TR \langle \chi
\rangle$ is a $\pi$-Hall subgroup of $\GL_{n_1}^\eta(q_1) \langle
\varphi\rangle$.

Now consider the natural homomorphism
$$\overline{\phantom{G}}:\GL_{n_1}^\eta(q_1)\leftthreetimes \langle
\varphi\rangle \rightarrow B, \text{ where }
B=\GL_{n_1}^\eta(q_1)\leftthreetimes \langle
\varphi\rangle/Z(\GL_{n_1}^\eta(q_1)).$$
Observe that $B$ is isomorphic to $\widehat{S}\Phi_S$, where
$\widehat{S}=\PGL_{n_1}^\eta(q_1)$ and $\Phi_S$ are defined in
Lemma~\ref{automorphisms}. By Lemma~\ref{L_Epi}(\ref{Epi}) we see that
$\overline{TR}$ is a  $\pi$-Hall subgroup of $\widehat{S}$ and
$\overline{H_1}$ is a $\pi$-Hall subgroup of $B$. Since
$|\widehat{S}:S|=\gcd(n_1,q_1-\eta)$ is a $\pi'$-number by Step~1(e), we
obtain $\overline{TR} \le S$ and $\overline{H_1}\cap S=\overline{TR}$ is a
$\pi$-Hall subgroup of $S$. In particular, $\overline{R}$ is a Sylow
$r$-subgroup of $S$.

Note that $H$ is a $\pi$-subgroup of $\Aut(S)$. It follows from
Lemma~\ref{automorphisms}  that $|\Aut(S)/B|\in\{1,2\}$. Since $2\notin \pi$,
we conclude that $H$ is contained in $B$. Since $H$ is a $\pi$-Hall subgroup
of $M$, we have that $H\cap S$ is a $\pi$-Hall subgroup of $S$.
Lemma~\ref{L_Epi}(\ref{epi=cpi}) yields that $S\in \C_\pi$. Therefore we may
assume that $H\cap S$ and $\overline{TR}$ coincide.

Step 1(f) implies that every Sylow $r$-subgroup of $S$ is a Sylow
$r$-subgroup of $\Aut(S)$. In particular, $\overline{R}$ is a Sylow
$r$-subgroup of $H$ and $H/(H\cap S)$ is a $\tau$-group. Recall that $r$ does
not divide $q_1-\eta$ by Step 1(e). Therefore, $\gcd(|R|,
|Z(\GL_{n_1}^\eta(q_1))|)=1$  and $R\simeq \overline{R}$. It now follows from
\cite[3.28]{Isaacs} that
$C_{\overline{T}}(\overline{R})=\overline{C_{T}({R})}$. Thus,
$$C_{H\cap
S}(\overline{R})=C_{\overline{TR}}(\overline{R})=\overline{R}C_{\overline{T}}
(\overline{R})=\overline{RC_T(R)}=\overline{C_{TR}(R)}.$$
Since $C_{TR}(R)\simeq (q_1-\eta)_\tau^{d+k}\times R$ by
Lemma~\ref{hall-gl}(e), where $d$ and $k$ are defined by $d=[n/r]$ and
$k=n-dr$, we obtain  that $\overline{C_{TR}(R)}\simeq
(q_1-\eta)_\tau^{d+k-1}\times R$ and $m_t(C_{H\cap S}(\overline{R}))=d+k-1$
for every $t\in \tau\cap \pi(S)$.

As we have seen above, $|\widehat{S}: S|$  is a $\pi'$-number. Thus
$$H/(H\cap S)=H/(H\cap \widehat{S})\simeq H\widehat{S}/\widehat{S}\le
B/\widehat{S}\simeq \langle\varphi\rangle$$
 and $H/(H\cap S)$ is cyclic. Therefore,  if $t\in\tau\cap \pi(S)$, then  $$m_t(C_H(\overline{R}))-m_t(C_{H\cap
S}(\overline{R}))\leq 1$$ and  $m_t(C_H(\overline{R}))$ is equal to either
$d+k-1$ or $d+k$.
 The Sylow theorems imply that
the same statement holds for the centralizer in $H$ of an arbitrary Sylow
$r$-subgroup of $H$.

Take some  $t\in \tau\cap \pi(S)$. As we have noted above, we complete Step 2
if we find a subgroup $E$ in $HS\le M$ such that $E$ is an elementary abelian
$t$-group of rank greater than $m_t(C_{H}(\overline{R}))$ and $E$ centralizes
a Sylow $r$-subgroup $R_0$ of~$HS$.

Lemma~\ref{elab-gl} implies that there is a subgroup $R_1\times K$ in
$\GL_{n_1}^\eta(q_1)$ such that $R_1$ is a Sylow $r$-subgroup of
$\GL_{n_1}^\eta(q_1)$ centralized by $ \varphi_{2'}$ and $K$ is a
$\varphi$-invariant elementary abelian $t$-subgroup
 of rank $2d+k$.
 A subgroup $\overline{K}$ is
 an elementary abelian $t$-subgroup of
$\widehat{S}=\overline{\GL_{n_1}^\eta(q_1)}$. Since  $|\widehat{S}: S|$  is a
$\pi'$-number, we conclude that
 $\overline{K}\le S$.

 If $m_t(C_H(\overline{R}))$ equals $d+k-1$, then  $E=\overline{K}$ is a desired
subgroup. Indeed, the rank of $\overline{K}$ is equal to $2d+k-1$ and is
greater than $d+k-1$, since $d=[n_1/r]>0$ in view of Step~1(e). Moreover,
$\overline{K}$  centralizes the Sylow $r$-subgroup $R_0=\overline{R_1}$
of~both
 $S$ and~$HS$.

If  $m_t(C_H(\overline{R}))$ equals $d+k$,
 then $|C_{H}(\overline{R})/C_{H\cap {S}}(\overline{R})|>1$ and
$C_H(\overline{R})$ contains an element $h$ of order $t$ such that $h\not\in
{S}$.  Moreover, $h\not\in \widehat{S}$,
 since  $|\widehat{S}: S|$  is a $\pi'$-number.
 In view of \cite[(7-2)]{GorLyons1983}
 we obtain that $\langle h\rangle=\langle\overline\psi\rangle^\delta$ where
 $\psi\in\langle\varphi\rangle$ is of order~$t$ and $\delta$ is an element in
$\widehat{S}$.
 By Lemma~\ref{elab-gl}(\ref{elab-gl_c}), we can assume that $K$ is centralized
by $\psi$.  The subgroup $E=\overline{\langle K, \psi\rangle}^\delta=
 \langle \overline{K}^\delta, \overline{\psi}^\delta\rangle$
 is an elementary abelian $t$-subgroup of $HS$. The rank of $E$ is equal to
$2d+k$ and $E$ centralizes
 the Sylow $r$-subgroup $R_0=\overline{R_1}^\delta$ of~both
 $S$ and~$HS$. So, $E$ is a desired subgroup.  This completes the proof of
Step~2.

\bigskip

{\bf Step 3.} $G$ is not a classical group.
\bigskip

  Assume that $G$ is a classical group, and so  $G$ satisfies either
Condition~II or Condition~III of Lemma~\ref{Dpi}. If $G$ satisfies either
item (g) or item (h) of Condition~II, then a $\pi$-Hall subgroup $H$ of $G$
is cyclic. Since $H\le M$, it follows from Lemma~\ref{L_Epi}(\ref{dpi_nil})
that $M$ satisfies $\D_\pi$, a contradiction. Therefore, $G$ satisfies either
Condition~III or one of items (a)-(f) of Condition~II, in particular
$e(q,t)=e(q,s)$ for every $t,s\in \tau$.

Set $$a=e(q,r) \mbox{ and } b=e(q,t) \mbox{ for every }t\in \tau.$$ Since $M$
is not almost simple by Step 2, the famous Aschbacher's theorem~\cite{Asch}
implies that  $M$ belongs to one of Aschbacher's classes
$\mathcal{C}_1-\mathcal{C}_8$. The structure of  members  of Aschbacher's
classes is specified in \cite{KleiLie}. Recall that by Step~1(b,d) $M$
possesses a composition factor $S\simeq \PSL_{n_1}^\eta(q_1)$, $r\in \pi(S)$
and $e(q_1,t)=e(q_1,s)$ for every $t,s\in \tau\cap\pi(S)$.  Set
$$a_1=e(q_1,r) \mbox{ and } b_1=e(q_1, t) \mbox{ for every }t\in \tau \cap
\pi(S).$$ Assume  that $q_1=q$. Then $a_1=a$ and $b_1=b$. Since $S$ satisfies
one of items \ref{L_epi-dpi_2B_a}-\ref{L_epi-dpi_2B_c} of
Lemma~\ref{epi-dpi_lie} by Step 1(d), we have that $a\neq b$ and $b\le 2$.
Consequently, $G$ cannot satisfy Condition~III, and so one of items (a)-(f)
of Condition~II holds for $G$. This implies that $b\ge r>2$, a contradiction.
Thus we conclude that $q_1\neq q$.

   We now consider Aschbacher's classes to specify all possibilities for $M$ to
have a composition factor $S$
   isomorphic to $\PSL_{n_1}^\eta(q_1)$ with $q_1\neq q$ (recall that $S$ cannot
be isomorphic to orthogonal or sympletic groups).
  The structure of members of Aschbacher's
   classes $\mathcal{C}_1-\mathcal{C}_8$  is  presented in
\cite[Chapter 4]{KleiLie}.
   By using this information, we check below that, in every case when $M$ is
an element of corresponding Aschbacher's
   class $\mathcal{C}_1-\mathcal{C}_8$, there is at most one such possibility
for $M$.

  \begin{itemize}
    \item[$\mathcal{C}_1$:] The structure of members of $\mathcal{C}_1$ is
        presented in \cite[\S4.1]{KleiLie}. The unique possibility for $M$
        appears in
    \cite[Proposition 4.1.18]{KleiLie}:

          (a) $G=\PSU_n(q)$, $M\simeq  \left\ldbrack q^{m(2n-3m)}
          \right\rdbrack \splitext \left\ldbrack c/\gcd(q+1,n)
\right\rdbrack\arbitraryext(\PSL_m(q^2)\times\\
          \phantom{(a)}\PSU_{n-2m}(q))\arbitraryext \left\ldbrack d
\right\rdbrack$,
          where $1\le m \le [n/2]$, \begin{gather*}
c=|\{(\lambda_1,\lambda_2)|\lambda_i\in \F_{q^2},
          \lambda_2^{q+1}=1,\lambda_1^{m(q-1)}\lambda_2^{n-2m}=1\}|,\\
           d=(q^2-1)\gcd(q^2-1,m)\gcd(q+1,n-2m)/c. \end{gather*} In this
           case, $S\simeq \PSL_{n_1}(q_1)$ with $n_1=m$ and $q_1=q^2$.

    \item[$\mathcal{C}_2$:] The structure of members of $\mathcal{C}_2$ is
        presented in \cite[\S4.2]{KleiLie}.
    The unique possibility for $M$ appears in \cite[Proposition
    4.2.4]{KleiLie}:

         (b) $G=\PSU_n(q)$, $M\simeq \left\ldbrack
\frac{(q-1)\gcd(q+1,\frac{n}{2})}{\gcd(q+1,n)}\right\rdbrack \arbitraryext
         \PSL_{n/2}(q^2)\arbitraryext \left\ldbrack
\frac{\gcd(q^2-1,\frac{n}{2})}{\gcd(q+1,\frac{n}{2})} \right\rdbrack
\arbitraryext 2 $.

         In this case, $S\simeq \PSL_{n_1}(q_1)$ with $n_1=n/2$ and
         $q_1=q^2$.

    \item[$\mathcal{C}_3$:] The structure of members of $\mathcal{C}_3$ is
        presented in
\cite[\S4.3]{KleiLie}. The unique possibility for $M$ appears in
    \cite[Proposition 4.3.6]{KleiLie}:

        (c) $G=\PSL_n^\eta(q)$, $M\simeq c   \arbitraryext \PSL_m^\eta(q^u)
\arbitraryext  d  \arbitraryext u$, where $n=mu$,
        $u$~is prime $\phantom{(ci)}$(if $\eta=-, u\ge 3$),
$c=\frac{\gcd(q-\eta,m)(q^u-\eta)}{(q-\eta)\gcd(q-\eta,n)}$,
        ${d=\frac{\gcd(q^u-\eta,m)}{\gcd(q-\eta,m)}}$.

        In this case, $S\simeq \PSL^\eta_{n_1}(q_1)$ with $n_1=m$,
$q_1=q^u$ and $\eta$ is the
        same for $G$ and $S$.

    \item[$\mathcal{C}_4, \mathcal{C}_7$:] The structure of members of
        $\mathcal{C}_4$ and $\mathcal{C}_7$
    presented in \cite[\S4.4 and \S4.7]{KleiLie} implies that if a
composition factor of members of $\mathcal{C}_4$ or  $\mathcal{C}_7$
    is isomorphic to $\PSL_{n_1}^\eta(q_1)$ then $q_1= q$.

    \item[$\mathcal{C}_5$:] The structure of members of $\mathcal{C}_5$ is
        presented in \cite[\S4.5]{KleiLie}.
    The unique possibility for $M$ appears in \cite[Proposition
    4.5.3]{KleiLie}:

        (d) $G=\PSL_n^\eta(q)$, $M$ is a normal subgroup in
$\PGL_n^\eta(q_1)$ of index $\frac{\lcm\left(
q_1-\eta,\frac{q-\eta}{\gcd(q-\eta,n)}\right) \gcd(q-\eta,n)}{q-\eta}$,
        where ${q=q_1^u}$, $u$~is prime and $u\ge 3$ if $\eta=-$.

        In this case, $S\simeq \PSL^\eta_{n_1}(q_1)$ with $n_1=n$ and
$\eta$ is the same for $G$ and~$S$.

    \item[$\mathcal{C}_6$:] The structure of members of $\mathcal{C}_6$
        presented in \cite[\S4.6]{KleiLie} implies that if a composition
        factor of members of $\mathcal{C}_6$
    is isomorphic to $\PSL_{n_1}^\eta(q_1)$ then $q_1= q$.

    \item[$\mathcal{C}_8$:] The structure of members of $\mathcal{C}_8$ is
        presented in \cite[\S4.8]{KleiLie}.
   The unique possibility for $M$ appears in \cite[Proposition
   4.8.5]{KleiLie}:

        (e) $G=\PSL_n(q)$, $M\simeq \PSU_n(q_1)\arbitraryext \left\ldbrack
        \frac{\gcd(q_1+1,n)c}{\gcd(q-1,n)}\right\rdbrack$, where
${q=q_1^2}$ and $c=\frac{q-1}{\lcm\left(
        q_0+1,\frac{q-1}{\gcd(q-1,n)}\right)}$.

        In this case, $S\simeq \PSL^-_{n_1}(q_1)$ with $n_1=n$.
     \end{itemize}

     In cases (d) and (e), $M$ is almost simple and we exclude them in view of
Step~2.

 Now we exclude the remaining cases (a)--(c). Recall some statements from Step~1
 which hold for $S$.

 \medskip
\begin{tabularx}{350pt}{XX}\hline
\multicolumn{1}{c}{$S\simeq \PSL_{n_1}(q_1)$}& \multicolumn{1}{c}{$S\simeq
\PSU_{n_1}(q_1)$}\\ \hline
\multicolumn{1}{c}{$a_1=r-1$}&   $r\equiv 1 \pmod{4}$ and $a_1=r-1$, \\
& or $r\equiv 3 \pmod{4}$ and $a_1=\frac{r-1}{2}$\\
\multicolumn{1}{c}{$b_1=1$ }& \multicolumn{1}{c}{$b_1=2$}\\
\multicolumn{2}{c}{$(q_1^{r-1}-1)_r=r$} \\
\multicolumn{2}{c}{$r\le n_1 \le r(r-2)$}\\ \hline
\end{tabularx}

\bigskip
If $S\simeq \PSL_{n_1}(q_1)$, we see that $a_1$ is even and $a_1>1$. If
$S\simeq \PSU_{n_1}(q_1)$, we see that either $a_1\equiv 0 \pmod{4}$, or
$a_1\equiv 1 \pmod{2}$. So, in the case where $S$ is unitary, $a_1$ cannot
equal $2k$ with $k$~odd, in particular, $a_1\neq 2$.

In the rest of Step~3, we fix some $t\in\tau\cap \pi(S)$.
\medskip

 {\it Cases} (a) and (b). In these cases
 $S\simeq \PSL_{n_1}(q^2)$. Since $e(q^2,t)=b_1=1$, we conclude $b=e(q,t)$ is
equal to 1 or 2.
 This implies that $G$ cannot satisfy Condition~II, and so $G$ satisfies
Condition~III; in particular $a=b$.
 Therefore, $a=e(q,r)$ equals 1 or 2, and $a_1=e(q^2,r)=1$, which is a
contradiction with the fact that $a_1>1$.

 {\it Case} (c). In this case $S\simeq \PSL_m^\eta(q^u)$, where $mu=n$ and
$u$~is prime (if $\eta=-$, $u\ge 3$).
 Show that $u=r$ and, in particular, $r$ divides~$n$.

 Assume first that $\eta=+$, i.~e. $G=\PSL_n(q)$ and $S\simeq \PSL_m(q^u)$.
Since $a_1=e(q^u,r)>1$, we have that $a=e(q,r)\geq a_1>1$
 and $a\neq u$. Since $b_1=e(q^u,t)=1$, we obtain that
$b=e(q,t)=\gcd(b,u)b_1=\gcd(b,u)$ divides $u$.
 Therefore, $b$ is equal to 1 or~$u$. Hence $G$ cannot satisfy Condition~III
where $a=b$, since $a$ cannot equal 1 or $u$.
 Consequently, $G$ satisfies either item~(a) or item~(b) of Condition~II, and
$b=r=u$.

 Assume now that $\eta=-$, i.~e. $G=\PSU_n(q)$ and $S\simeq \PSU_m(q^u)$.
 Since
 $$
 a=\gcd(a,u)a_1\quad\text{ and }\quad a_1\not\equiv 2\pmod 4,
 $$
 we have that $a\not\equiv 2\pmod 4$. In particular $a\ne 2$ and $a\ne 2u$
(recall that $u\ge 3$ is prime in this case).
  It follows from $b=\gcd(b,u)b_1$ and $b_1=2$ that $b$ is equal to 2 or $2u$.
Consequently, $a\ne b$ and
  $G$ satisfies one of items (c)-(f) of Condition~II. This implies $b=2r=2u$.

Thus we have that $r=u$, and so $r$ divides $n$. Therefore, $G$ cannot
satisfy items~(b), (e) and (f) of Condition~II, where $n \equiv -1 \pmod{r}$.
Hence $G$ satisfies one of items~(a), (c) or (d) of Condition~II. Now it
follows that
$$\left[\frac{n}{r-1}\right]=\left[\frac{n}{r}\right]=\frac{n}{r}=m.$$ These
equalities yield that
$$n=mr=m(r-1)+m$$ and $m<r-1$, which is a contradiction with the fact that
$m=n_1\ge r$.

 \bigskip

Thus in all cases (a)-(e) we obtain a contradiction, and so $G$ cannot be a
classical group, as wanted. To prove the statement \eqref{ms} it remains to
show that  $G$ cannot be an exceptional group.

 \bigskip
{\bf Step 4.}  $G$ is not an exceptional group.
 \bigskip

 Assume that $G$ is an exceptional group, and so $G$ satisfies either
Condition~III or Condition~IV of Lemma~\ref{Dpi}. The description of
$\pi$-Hall subgroups in the exceptional groups in characteristic $p$ with
$2,p\notin \pi$ is given in \cite{VdovinRevin2002}. Recall that by Step~1 $M$
possesses a composition factor $S\simeq \PSL_{n_1}^\eta(q_1)$, $r\le n_1\le
r(r-2)$ and $|\tau\cap \pi(S)|\geq 1$. Also if $\eta=+$ then $e(q_1,t)=1$,
and if $\eta=-$ then $e(q_1,t)=2$ for every $t\in \tau\cap\pi(S)$, in
particular, $q_1\ge 4$.

 Suppose $G$ satisfies Condition~III first. So $G$ is  not isomorphic
to ${}^2B_2(q)$, ${}^2G_2(q)$ or ${}^2F_4(q)'$. Hence by \cite[Lemmas
7-13]{VdovinRevin2002} we have that a $\pi$-Hall subgroup $H$ of $G$ is
abelian or $\pi\cap\pi(G)\subseteq \pi(q\pm1)$. If $H$ is abelian then $M$
satisfies $\D_\pi$ by Lemma~\ref{L_Epi}(\ref{dpi_nil}), a contradiction.
Consequently  $\pi\cap\pi(G)\subseteq \pi(q\pm1)$. If $q_1=q^u$ for some
natural $u$, it follows from $\pi\cap\pi(S)\subseteq\pi(q\pm 1)$ that
$\pi\cap\pi(S)\subseteq\pi(q^u\pm 1)=\pi(q_1\pm 1)$, and so
$e(q_1,t)=e(q_1,s)$ for every $t,s \in \pi\cap\pi(S)$. Since $(S,\pi)$
satisfies \eqref{*} by Step~1(c), Lemma~\ref{dpi_not2} implies that $S$
satisfies $\D_\pi$, a contradiction. Thus we conclude that $q_1$ is not a
power of $q$.

 Let $\overline{G}$ be a adjoint simple algebraic
group and $\sigma$ a surjective endomorphism of $\overline{G}$ such that
$G=O^{p'}(\overline{G}_\sigma)$. All maximal subgroups of $G$ are presented
in Lemma~\ref{max_excep}. Since $M$ is not almost simple, we consider all
possibilities for $M$ according to items (a)-(e) of Lemma~\ref{max_excep}.

\medskip
 {\it Case} (a):  $M=N_G(D_\sigma)$, where $D$ is a $\sigma$-stable closed
connected subgroup and $D$ is either parabolic or
 reductive subgroup of maximal rank. If $D$ is parabolic, then there are no
composition factors of $M$ isomorphic to
 $\PSL_{n_1}^\eta(q_1)$ with $q_1\neq q^u$. If $D$ is reductive subgroup of
maximal rank, then $M$ is a subgroup of maximal rank  in sense
of~\cite{LieSaxlSeitz1992}. Since $G\unlhd \overline{G}_\sigma$, we have that
$S$ is a composition factor of $N_{\overline{G}_\sigma}(D_\sigma)$.
 According to Tables 5.1 and 5.2 from \cite{LieSaxlSeitz1992}, we obtain that
$S$ is isomorphic to one of the following groups $\PSL_2(5)$, $\PSL_3(2)$ or
$\PSU_4(2)$.
 If $S\simeq \PSL_2(5)$ then $n_1=2$, which is a contradiction with the fact
that $n_1\ge r>2$. If $S$ is isomorphic to $\PSL_3(2)$ or $\PSU_4(2)$ then
$q_1=2$ and it  contradicts the fact $q_1\ge 4$.

\medskip
 {\it Case} (b): $M=N_G(E)$, where $E$ is an elementary abelian $s$-subgroup
with $s$ prime and $E\le \overline{G}_\sigma$ (see Table 2).
 Since $G\unlhd \overline{G}_\sigma$, we have that $S$ is a composition factor of
$N_{\overline{G}_\sigma}(E)$.
 According to Table~\ref{Table2}, we obtain that $S$ is isomorphic to one of the
following groups $\PSL_3(2)$, $\PSL_3(3)$, $\PSU_3(3)$, $\PSL_3(5)$ or
$\PSL_5(2)$.
 If $S$ is isomorphic to $\PSL_3(2)$, $\PSL_3(3)$, $\PSU_3(3)$ or $\PSL_5(2)$,
then $q_1<4$, a contradiction.
 If $S\simeq \PSL_3(5)$, then there is no odd prime $t$ with $e(q_1,t)=1$, and
so $\tau\cap\pi(S)=\varnothing$,
 which is a contradiction with the fact that $|\tau\cap\pi(S)|\geq 1$.

 \medskip
 {\it Case} (c): $M$ is the centralizer of a graph, field, or graph-field
automorphism of $G$ of prime order. The structure of $M$ is presented
 in \cite[Theorem 4.5.1, Theorem 4.7.3, Propositions 4.9.1 and 4.9.2]{CFSG}.
 As we mentioned in Remark on Lemma \ref{epi-dpi_lie}, $S$ cannot be isomorphic
to an orthogonal or sympletic group. So we see that there are no centralizers
of a graph, field, or graph-field automorphism of $G$ of prime order with a
composition factor isomorphic to $\PSL_{n_1}^\eta(q_1)$ where $q_1\neq q^u$.

  \medskip
 {\it Cases} (d) and (e): either $\overline{G}=E_8$, $p>5$ and $F^*(M)$ is one
of groups $\PSL_2(5)\times\PSL_2(9)$, $\PSL_2(5)\times \PSL_2(q)$ or
 $F^*(M)$ is as in    Table \ref{Table3}. The maximality and the structure of
$M$
 implies that $M$ is a subgroup of $\Aut(F^*(M))$, in particular, $M/F^*(M)$ is
solvable. Hence case (d) holds and $S$ is isomorphic to one of the following
groups
 $\PSL_2(5)$, $\PSL_2(9)$. Now we have $n_1=2$, a contradiction with the fact
that $n_1\ge r>2$.

 Thus in all cases (a)-(e) we obtain a contradiction, and so $G$ cannot satisfy
Condition~III.
 Consequently, we conclude that $G$ satisfies Condition~IV, and $G$ is
isomorphic to one of groups ${}^2B_2(q)$, ${}^2G_2(q)$ or ${}^2F_4(q)'$.
 Since $2\notin \pi$, in view of \cite[6.13 Corollary]{Gross1986} ${}^2F_4(2)'$
does not satisfy $\E_\pi$, and therefore $G$ cannot be isomorphic
 to ${}^2F_4(2)'$. Since ${}^2F_4(q)'={}^2F_4(q)$ with $q>2$, further we write
${}^2F_4(q)$ instead of ${}^2F_4(q)'$.

  By  \cite[Lemma 14]{VdovinRevin2002}, if $G$ is isomorphic to ${}^2B_2(q)$ or
${}^2G_2(q)$, or if $G$ is isomorphic
  to ${}^2F_4(q)$ and  $3\notin \pi$, then $H$ is abelian. Since $H\le M$, it
now follows form Lemma~\ref{L_Epi}(\ref{dpi_nil})
  that $M$ satisfies $\D_\pi$, a contradiction. Consequently, we deduce that $G$
is isomorphic to ${}^2F_4(q)$ and $3\in\pi$,
  and therefore $r=3$.  All maximal subgroups of $G$ are specified in
Lemma~\ref{max_2F4}. Since $M$ has a composition factor
  $S\simeq \PSL_{n_1}^\eta(q_1)$ and $r\le n_1 \le r(r-2)$, we obtain that
$n_1=3$ and $S\simeq\PSU_3(q_1)$ with $q_1=q=2^{2m+1}$.
  By Step 1(d) $S$ satisfies \ref{L_epi-dpi_2B_c} of Lemma~\ref{epi-dpi_lie}. In
particular, $$e(q_1,r)=\frac{r-1}{2}=1.$$
  But  $q_1=2^{2m+1} \equiv -1 \pmod 3.$ So we obtain a contradiction with
  the fact that $e(q_1,3)=e(q_1,r)=1$.

  Thus we conclude that $G$ cannot be an exceptional group, and so the main
theorem is proved.

\section*{References}


\begin{thebibliography}{}

\bibitem{Asch} Aschbacher	M.: On the maximal
    subgroups of the finite classical groups. Inventiones mathematicae 76 (3)
    469–-514 (1984).
\bibitem{Borel1970} Borel A. and Institute for Advanced
    Study (Princeton, N.J.): Seminar on algebraic groups and related finite
    groups. Lecture notes in mathematics, Springer-Verlag (1970).
\bibitem{GorLyons1983} Gorenstein D.,
    Lyons R.: The local structure of finite groups of characteristic 2 type.
    Vol. 42, American Mathematical Society (1983).
\bibitem{CFSG} Gorenstein D.,
    Lyons R., Solomon R.: The classification of the finite simple groups. Number
    3, American Mathematical Soc., Providence, RI (1998).
\bibitem{Gross1986} Gross F.: On a conjecture of Philip
    Hall. Proc. London Math. Soc. s3-52 (3),
    464–-494  (1986).
\bibitem{Gross1987} Gross F.: Conjugacy of odd order Hall
    subgroups. Bull. London Math. Soc. 19 (4),
    311–-319  (1987).
\bibitem{Gross1995} Gross F.: Odd order Hall subgroups of
    the classical linear groups. Mathematische Zeitschrift 220 (1),
    317–-336  (1995).
\bibitem{Hall1956} Hall P.: Theorems like Sylow’s.
    Proceedings of the London Mathematical Society s3-6 (2), 286–-304  (1956).
\bibitem{Isaacs} Isaacs I.M.:  Finite group theory.
    Graduate Studies in Mathematics, Vol. 92, American Mathematical Soc.,
    Providence, RI  (2008).
\bibitem{KargMerz} Kargapolov M.I.,
    Merzljakov Ju.I.:  Fundamentals of the theory of groups.  Graduate Texts
    in Mathematics, 62. Springer-Verlag, New York-Berlin (1979).
\bibitem{KleiLie} Kleidman P.B.,
    Liebeck M.W.: The subgroup structure of the finite classical groups. Vol. 129,
    Cambridge University Press (1990).
\bibitem{LieSaxlSeitz1992}   Liebeck M.W., Saxl J., Seitz G.M.: Subgroups of
    maximal rank in finite
    exceptional groups of Lie type. Proc. London Math. Soc 65 (3)
    297–-325  (1992).
\bibitem{LieSeitz} Liebeck M.W., Seitz G.M.:
    Maximal subgroups of exceptional groups of Lie type, finite and
    algebraic. Geom. Dedicata 36, 353–-387 (1990).
\bibitem{Malle1991} Malle G.: The maximal subgroups of
    ${}^2F_4(q^2)$. J.Algebra 139 (1), 52-–69 (1991).
\bibitem{Manzaeva2014} Manzaeva N.Ch.: Heritability of
    the property $\D_\pi$ by overgroups of $\pi$-Hall subgroups in the case
    where $2\in \pi$. Algebra and Logic 53 (1), 17–-28 (2014).
\bibitem{Kour} Mazurov 	V.D., Khukhro E.I. (Eds.): The Kourovka notebook.
    Unsolved problems in group theory. 17th Edition, Russian Academy of
    Sciences Siberian Division Institute of Mathematics, Novosibirsk  (2010).
\bibitem{Nesterov} Nesterov M.N.: On pronormality and
    strong pronormality of Hall subgroups. Siberian Mathematical Journal 58
    (1), 128–-133 (2017).
\bibitem{RevinVdovin2006} Revin D.O.,
    Vdovin E.P.: Hall subgroups of finite groups. Contemporary Mathematics 402
    229–-265 (2006).
\bibitem{Revin2008e} Revin D.O.: The $\D_\pi$-property in
    finite simple groups. Algebra and Logic 47 (3), 210–-227 (2008).
\bibitem{Revin2009around} Revin D.O.: Around a conjecture
    of P. Hall. Siberian Electronic Mathematical Reports 6, 366–-380  (2009)
    (in Russian).
\bibitem{Steinberg} Steinberg R.: Automorphisms of
    finite linear groups. Canad. J. Math 12 (4), 606–-616 (1960).
\bibitem{Steinberg1968} R. Steinberg, Endomorphisms of
    linear algebraic groups, Mem. Amer. Math. Soc. 80 (1968).
\bibitem{VdovinManRevin2012}
    Vdovin E.P., Manzaeva N.Ch., Revin D.O.: On the heritability of the
    property $\D_\pi$ by subgroups. Proceedings of the Steklov Institute of
    Mathematics 279 (1) 130–-138 (2012).
\bibitem{VdovinRevin2002} Vdovin E.P.,
    Revin D.O.: Hall subgroups of odd order in finite groups. Algebra and Logic 41
    (1), 8--29 (2002).
\bibitem{VdovinRevin2011} Vdovin E.P.,
    Revin D.O.: Theorems of Sylow type. Russian Math. Surveys 66 (5),
    829–-870 (2011).
\bibitem{VdovinRevin2012} Vdovin E.P.,
    Revin D.O.: Pronormality of Hall subgroups in finite simple groups. Siberian
    Math. J. 53 (3), 419–-430 (2012).
\bibitem{VdovinRevin2013} Vdovin E.P.,
    Revin D.O.: On the pronormality of Hall subgroups. Siberian Math. J. 54
    (1),     22–-28 (2013).
\bibitem{Weir1955} Weir A.J.: Sylow $p$-subgroups of the
    classical groups over finite fields with characteristic prime to $p$.
    Proc. AMS 6 (4) 529–-533 (1955).
\bibitem{Wielandt1954} Wielandt H.: Zum Satz von Sylow.
    Math. Z. 60 (1), 407–-408 (1954).
\end{thebibliography}


\end{document}